\documentclass[12pt]{amsart}

\usepackage[all]{xy}
\usepackage{fullpage}
\usepackage{latexsym}
\usepackage{amsmath}
\usepackage{amsfonts}
\usepackage{amssymb}
\usepackage{amsthm}
\usepackage{eucal}
\usepackage{enumerate,yfonts}
\usepackage{mathrsfs}
\usepackage{graphicx}
\usepackage{graphics}
\usepackage{epstopdf}
\usepackage{amscd}
\usepackage{bbm}
\usepackage{hyperref}
\usepackage{url}
\usepackage{color}
\usepackage{bbm}
\usepackage{cancel}
\usepackage{enumerate}
\usepackage{amsmath,amsthm}
\usepackage{amssymb}
\usepackage{epsfig}
\usepackage{pstricks}
\usepackage{xy}
\usepackage{xypic}
\usepackage{tikz}

\newtheorem{thm}{Theorem}[section]
\newtheorem{corollary}[thm]{Corollary}
\newtheorem{lemma}[thm]{Lemma}

\newtheorem{prop}[thm]{Proposition}
\newtheorem{conjecture}[thm]{Conjecture}
\newtheorem{thm-dfn}[thm]{Theorem-Definition}

\theoremstyle{definition}
\newtheorem{definition}[thm]{Definition}

\numberwithin{equation}{section}

\theoremstyle{remark}
\newtheorem{remark}{Remark}[section]
\newtheorem{example}[remark]{Example}

\newcommand{\fg}{{\mathfrak g}}
\newcommand{\ft}{{\mathfrak t}}

\newcommand{\fb}{{\mathfrak b}}

\newcommand{\fp}{{\mathfrak p}}

\newcommand{\fa}{{\mathfrak a}}

\newcommand{\fc}{{\mathfrak c}}
\newcommand{\fk}{{\mathfrak k}}
\newcommand{\fh}{{\mathfrak h}}
\newcommand{\fq}{{\mathfrak q}}
\newcommand{\fm}{{\mathfrak m}}

\newcommand{\rW}{{\mathrm W}}

\newcommand{\bC}{{\mathbb C}}

\newcommand{\bG}{{\mathbb G}}

\newcommand{\mE}{\mathcal{E}}
\newcommand{\mF}{\mathcal{F}}

\newcommand{\on}{\operatorname}
\newcommand{\D}{\on{D}}

\newcommand{\ra}{\rightarrow}

\newcommand{\is}{\simeq}

\newcommand{\Loc}{\on{Loc}}
\newcommand{\Bun}{\on{Bun}}

\newcommand{\quash}[1]{}  
\newcommand{\nc}{\newcommand}

\newcommand{\bbC}{{\mathbb C}}

\newcommand{\bbP}{{\mathbb P}}

\newcommand{\bbR}{{\mathbb R}}

\newcommand{\calA}{{\mathcal A}}

\newcommand{\calC}{{\mathcal C}}

\newcommand{\calE}{{\mathcal E}}
\newcommand{\calF}{{\mathcal F}}
\newcommand{\calG}{{\mathcal G}}
\newcommand{\calH}{{\mathcal H}}

\newcommand{\calM}{{\mathcal M}}
\newcommand{\calN}{{\mathcal N}}
\newcommand{\calO}{{\mathcal O}}

\newcommand{\Sym}{\on{Sym}}

\nc{\al}{{\alpha}} \nc{\be}{{\beta}} \nc{\ga}{{\gamma}}
\nc{\ve}{{\varepsilon}} \nc{\Ga}{{\Gamma}} 
\nc{\La}{{\Lambda}}
\nc{\ho}{\text{hol}}

\nc{\ad }{{\on{ad }}}

\nc{\aff}{{\on{aff}}} \nc{\Aff}{{\mathbf{Aff}}}

\nc{\der}{{\on{der}}}

\nc{\diag}{{\on{diag}}}

\nc{\Fl}{{\calF\ell}}

\newcommand{\Gr}{{\on{Gr}}}
\nc{\Hg}{{\on{Higgs}}}

\nc{\Id}{{\on{Id}}}

\nc{\Ind}{{\on{Ind}}}
\newcommand{\Lie}{{\on{Lie}}}
\nc{\Op}{{\on{Op}}}

\nc{\res}{{\on{res}}}

\nc{\tr}{{\on{tr}}}

\nc{\GSp}{{\on{GSp}}} \nc{\GU}{{\on{GU}}} \nc{\SL}{{\on{SL}}}
\nc{\SU}{{\on{SU}}} \nc{\SO}{{\on{SO}}}

\nc{\nh}{{\Loc_{J^p}(\tau')}}
\nc{\bnh}{{\Loc_{\breve J^p}(\tau')}}

\nc{\bU}{{\overline{U}}} \nc{\IC}{{\on{IC}}}
\newcommand{\Nm}{{\on{Nm}}}

\newcommand{\ct}{\check\theta}

\newcommand{\beqn}{\begin{equation*}}
\newcommand{\eeqn}{\end{equation*}}

\newcommand{\beq}{\begin{equation}}
\newcommand{\eeq}{\end{equation}}

\tikzset{node distance=2em, ch/.style={circle,draw,on chain,inner sep=2pt},chj/.style={ch,join},every path/.style={shorten >=4pt,shorten <=4pt},line width=1pt,baseline=-1ex}

\quash{
\setlength{\parskip}{1ex}
\setlength{\oddsidemargin}{0in}
\setlength{\evensidemargin}{0in}
\setlength{\textwidth}{6.5in}
\setlength{\topmargin}{-0.15in}
\setlength{\textheight}{8.6in}
\topmargin-0.5cm \textheight22cm \oddsidemargin1.2cm \textwidth14cm}
\begin{document}
\title{
A relative Langlands dual realization of $T^*(G/K)$
and derived Satake
}

         \author{Tsao-Hsien Chen}
        \address{School of Mathematics, University of Minnesota, Twin cities, Minneapolis, MN 55455 }
         \email{chenth@umn.edu}
\thanks{}
\thanks{}

\maketitle   
\begin{abstract}
We show that the cotangent bundle $T^*(G/K)$ of
a quasi-split symmetric space $G/K$
is isomorphic to  the  dual variety of the loop
symmetric space for 
 the Langlands dual group, providing instances of the relative Langlands duality for non-split groups.
Then we establish a Langlands dual  description of 
equivariant coherent sheaves on $T^*(G/K)$ in terms of constructible sheaves on 
the loop symmetric spaces,
generalizing the 
derived Satake equivalence 
for reductive groups to quasi-split symmetric spaces.
To this end, we prove the 
derived Satake equivalence for the twisted affine Grassmannians, study ring objects arising from  loop symmetric spaces, 
and explore the formality and fully-faithfulness properties of $!$-pure objects. 
We 
deduce a version of Bezrukavnikov equivalence for quasi-split symmetric spaces and 
make connections to 
the geometric Langlands on
the twistor $\mathbb P^1$.

\end{abstract} 

\setcounter{tocdepth}{1} 
\setcounter{tocdepth}{2} 
\tableofcontents
\section{Introduction}
\subsection{Main results}
Let $G$ and $\check G$ be Langlands dual complex connected reductive groups.
Let $F=\bC((t))$ and $\calO=\bC[[t]]$
be the formal Laurent and Taylor series rings.
In \cite{BZSV} and \cite{BFN1,BFN2} the authors observe 
that one can associate to each  affine smooth $\check G$-variety $\check X$, an (super) affine Hamiltonian $G$-variety $M_{\check G,\check X}$
to be called the relative  dual (or $S$-dual, symplectic dual)  of $\check X$ with several remarkable properties. 
The construction of the relative  dual variety 
$M_{\check G,\check X}$, or rather, its coordinate ring $\bC[M_{\check G,\check X}]$, uses the idea of ring objects in 
the derived Satake category  
(or equivalently, the de equivariantized Ext-algebra):
associated with 
the loop space $\check X(F)$ of $\check X$, there is a ring object 
$\calA_{\check G,\check X}$
in the derived Satake category
$ D^b_{\check G(\calO)}(\Gr_{\check G})$ for $\check G$ such that there is a ring isomorphism
\[
    \bC[M_{\check G,\check X}]\cong \on{H}^*_{\check G(\calO)}(\Gr_{\check G},\calA_{\check G,\check X}\otimes^!\IC_{\bC[ 
G]}).
\]
Here $\IC_{\bC[G]}$
is the regular perverse sheaf on the affine Grassmannian 
$\Gr_{\check G}=\check G(F)/\check G(\calO)$ of $\check G$.
In the case when 
$\check X$ is a spherical variety (satisfying certain assumptions) the relative Langlands duality \cite{BZSV} proposes an elementary construction of $M_{\check G,\check X}$ 
from the combinatorics of $\check X$
and predicts deep connections between 
the algebraic geometry of $M_{\check G,\check X}$ and the harmonic analysis and representation theory of the spherical variety 
$\check X$.
A fundamental example is the case when
$\check X=\check H$ is a connected reductive group viewed 
as a spherical variety for $\check G=\check H\times\check H$.\footnote{Here we consider the twisted action of $H\times H$
on $H$ given by
$(h_1,h_2)\cdot h=h_1h(h_2^{-1})'$
where $h_2'$ is the image of $h_2$ under
the Chevalley involution.} 
The work \cite{ABG} implies that the coordinate ring  of the 
relative dual variety $M_{\check H\times \check H,\check H}$ of $\check H$
is isomorphic to that of the cotangent bundle $T^*H$ of  $H$
\begin{equation}\label{relative dual}
    \bC[M_{\check H\times \check H,\check H}]\cong\bC[T^*H] 
\end{equation}
and the 
derived Satake equivalence \cite{BF}, a categorificaiton of the Harish-Chandra Plancherel formula, provides a spectral description
\begin{equation}\label{DS}
   D^b_{\check H(\calO)}(\Gr_{\check H})\cong  \on{Perf}^{H\times H}_{}(T^*H[2]) 
\end{equation}
of  the  derived Satake category of $\check H$ in terms of  $H\times H$-equivariant perfect complexes on the shifted cotangent bundle $T^*H[2]$, viewed as a dg-scheme with trivial differential.

There has been a lot of  work on 
generalizing~\eqref{relative dual} and~\eqref{DS}
to a more general $\check X$, see, e.g., the works \cite{BFN1,BFN2,BFGT,CN1,CMNO,De1,M}, etc. 
In the
case when $\check X$ is a homogeneous affine $\check G$-variety with connected stabilizers, that is, $\check X\cong \check G/\check H$ where $\check H\subset\check G$ is an arbitrary connected redcutive subgroup,   
the recent work \cite{G2}
establishes an elementary description of the relative dual variety $M_{\check G,\check X}$ 
in terms of the regular centralizer group schemes of $G$ 
and $H$.
It is expected that 
when $\check X$ is spherical 
the description in \emph{loc. cit.} agrees with the construction in \cite{BZSV} 
(in a highly non-obvious way, see, e.g., \cite{CDMN}).

In this notes we will consider 
an inverse question of above story. 
Namely, instead of trying to 
find the relative dual $M_{\check G,\check X}$ 
of a given smooth affine $\check G$-variety $\check X$, we will start with an affine Hamiltonian $G$-variety 
$M$ of representation theory or algebraic geometry interests and ask whether it is the relative dual variety 
$M\cong M_{\check G,\check X}$
for some $\check G$-variety $\check X$.
From the perspective of representation theory of  real groups (see Section \ref{twistor}),
the cotangent bundle $M=T^*X$ of a 
symmetric space $X=G/K$ for $G$ would be an interesting and important example to investigate.
In this note we provide a positive answer to this question assuming $G$ is adjoint and 
 $X$ is quasi-split.
Surprisingly, we will see non-split  groups over local fields and their symmetric spaces.
To describe the answer, 
let $\theta\in\on{Aut}(G)$ be the (complex) Cartan involution on $G$ such that $K=G^\theta$.
Let $\check\theta_0$ be a pinned involution of 
$\check G$ that is dual to $\theta$ in the sense that  $\theta$ is in the inner class of  $\on{Chev}\circ\theta_0$,
where $\theta_0$ is a pinned involution of $\check G$ with the same image as
$\theta_0$ in the outer automorphism group 
$\on{Out}(\check G)\cong\on{Out}(G)$ and 
$\on{Chev}$ is the Chevalley involution of $G$ (see \eqref{theta}).
Denote by 
$\check H=\check G^{\check\theta_0}$ the corresponding symmetric subgroup.
A natural candidate for $\check X$ would be 
the symmetric space $\check X=\check G/\check H$. It works in the group case 
but does not give the correct space in general.
For example,
if $X$ is a split symmetric space then  $\theta_0$, and hence
$\check\theta_0$, is trivial and $\check X$ is just a point.
Instead we will consider the
involution 
$\check\theta:\check G(F)\to\check G(F)$ 
on the loop group $\check G(F)$ sending 
$\gamma(t)$ to
$\ct(\gamma)(t)=\check\theta_0(\gamma(-t))$.
The $\ct$-fixed points subgroup 
$\check G(F)^{\ct}$ is the so-called twisted loop group and it turns out that the correct space to consider is the quotient
\[\check X_{\ct}(F)=\check G(F)/\check G(F)^{\ct}\]
to be called 
the  loop symmetric space for 
$\check G(F)$. 
In the case when $X$ is a split symmetric space,
we have $\ct(\gamma)(t)=\gamma(-t)$ and hence
 $\check X_{\ct}(F)=\check G(F)/\check G(F')$
 where $F'=\bC((t^2))\subset F$ is the subfield of formal Laurent series  in $t^2$.
For a list of loop symmetric spaces
associated to quasi-split but non-split $X$, see Example \ref{list 3}.

We observe that the ring objects construction 
of the
dual varieties (or the de equivariantized Ext-algebras constructions) in \cite{BFN2,BZSV} 
works equally well for the loop symmetric space:
associated with $\check X_{\ct}(F)$,
there is a ring object 
$\calA_{\check G,\ct}$
in the derived Satake category of $\check G$ and we define the dual variety $M_{\check G,\ct}$ of the loop symmetric space $\check X_{\ct}(F)$ to
be the (super) affine scheme with coordinate ring 
\[\bC[M_{\check G,\ct}]=\on{H}^*_{\check G(\calO)}(\Gr_{\check G},\calA_{\check G,\ct}\otimes^!\IC_{\bC[G]})\]
(see Definition \ref{dual variety}).
We now can state our first main result: 
\begin{thm}\label{intro:main 1}
    There is a natural 
    $G$-equivariant  
    isomorphism of rings 
\begin{equation}\label{intro: main 1}
    \bC[M_{\check G,\ct}]\cong\bC[T^*X] 
\end{equation}
\end{thm}

Theorem \ref{intro:main 1} is restated in Theorem \ref{dual construction}.
We refer to Section \ref{TRS} for a more
detailed explanation of the statement including the compatibility with the moment maps.
The isomorphism above can be regarded as a generalization of~\eqref{relative dual} from groups 
to all quasi-split symmetric varieties.
Note that we can regard
$\check G(F)^{\ct}$ as a symmetric subgroup for the Weil restriction 
$\on{Res}^F_{F'}\check G(F)$, viewed as a non-split group over $F'$, and $\check X_{\ct}(F)$ as a symmetric space $\on{Res}^F_{F'}\check G(F)/\check G(F)^{\ct}$ for $\on{Res}^F_{F'}\check G(F)$.
From this point of view,
Theorem \ref{intro:main 1} provides a family of concrete examples of relative dual varieties of 
symmetric spaces for non-split groups.

Our second main result extends the derived Satake equivalence in~\eqref{DSat} to quasi-split symmetric spaces:
consider the derived Satake category 
for $\check X_{\ct}(F)$, that is, the full subcategory 
$D^b_{\check G(\calO)}(\check X_{\ct}(F))_e\subset D^b_{\check G(\calO)}(\check X_{\ct}(F))$
of the $\check G(\calO)$-equivariant derived category of $\check X_{\ct}(F)$
generated by the 
dualizing sheaf $\omega_{e}$ of the 
closed $\check G(\calO)$-orbit $\check G(\calO)\cdot e$ through the base point $e\in\check X_{\ct}(F)$
under the Hecke action.
\\

\begin{thm}\label{intro:main 2}
There is a canonical equivalence
  \begin{equation}\label{intro: equ main 2}
 D^b_{\check G(\calO)}(\check X_{\ct}(F))_e \cong \on{Perf}^G(T^*X[2])
 \end{equation}
 compatible with the monoidal action 
 of derived Satake equivalence for $\check G$ on both sides.
\end{thm}
Theorem \ref{intro:main 2} is restated in Theorem \ref{coherent sheaves}. We refer to Section \ref{TRS} for a more
detailed explanation of the statement.
The equivalence~\eqref{intro: equ main 2}
provides a constructible realization 
of the category 
$G$-equivariant perfect complexes on the shifted cotangent bundle $T^*X[2]$.
In \emph{loc. cit},
we also establish variants of the equivalence including 
a constructible realization of the full subcategory $\on{Perf}^G(T^*X[2])_{\calN_{\fp^*}}$
of perfect complexes set theoretically supported on the preimage of symmetric nilpotent cone $\calN_{\fp^*}$
under the moment map.

The proof of Theorem \ref{intro:main 1} and Theorem \ref{intro:main 2} relies on the Kostant-Rallis theory \cite{KR} and the study of ring objects in the derived Satake category $D^b_{\check G(\calO)^{\ct}}(\Gr_{\check G}^{\ct})$ for the
 twisted affine Grassmannian 
$\Gr^{\ct}_{\check G}=\check G(F)^{\ct}/\check G(\calO)^{\ct}$. The main technical results include 
 a version of  derived Satake equivalence for $\Gr_{\check G}^{\ct}$, 
extending the work of Zhu \cite{Z}
to the derived setting, 
and a 
fully-faithfullness result for 
$!$-pure objects in $D^b_{\check G(\calO)^{\ct}}(\Gr_{\check G}^{\ct})$
generalizing the work in \cite{G2}:
\begin{thm}\label{intro:main 3}
(1) There is a canonical monoidal equivalence
 \begin{equation}\label{intro:DSt}
      D^b_{\check G(\calO)^{\ct}}(\Gr^{\ct}_{\check G}) \cong \on{Perf}^{H\times H}(T^*H[2])
 \end{equation}
  here $H=G^{\theta_0}$ is the fixed points subgroup of the pinned involution $\theta_0$.

 (2)  Let $\calF\in\on{Ind}(D^b_{\check G(\calO)}(\Gr_{\check G}))$ be a very pure object and $\calE\in\on{Ind}(D^b_{\check G(\calO)^{\ct}}(\Gr_{\check G}^{\ct}))$ be a $!$-pure object.
 Let  
$\check\rho:\Gr_{\check G}^{\ct}\to\Gr_{\check G}$
be the natural inclusion.
Then there is a natural isomorphism
\begin{equation}\label{intro:twisted Ginzburg}
    \on{H}^*_{\check G(\calO)^{\ct}}(\Gr^{\ct}_{\check G},\calE\otimes^!\check\rho^!\calF)
    \cong (\on{H}^*_{\check G(\calO)^{\ct}}(\Gr^{\ct}_{\check G},\calE)\otimes_{\bC[\fc_H]}\on{H}^*_{\check G(\calO)^{\ct}}(\Gr_{\check G},\calF))^{J_G|_{\fc_H}}.
\end{equation}
Here $J_G|_{\fc_H}$ acts on 
the tensor product through the natural action on $\on{H}^*_{\check G(\calO)^{\ct}}(\Gr_{\check G},\calF)$
and the pull-back action of $J_H$ on $\on{H}^*_{\check G(\calO)^{\ct}}(\Gr^{\ct}_{\check G},\calE)$ via the norm map
$\on{Nm}:J_G|_{\fc_H}\to J_H$.
If $\calF$ and $\calE$ are ring objects then~\eqref{intro:twisted Ginzburg} is a ring isomorphism.
\end{thm}

Theorem \ref{intro:main 3}
is the combination of Theorem \ref{DSat} and Theorem \ref{full-faithfulness}. 
We refer to Section \ref{twisted Satake} and \ref{TRS} for a more
detailed explanation of the statement. The key steps in the proof of~\eqref{intro:DSt} 
include the calculation of equivariant homology of $\Gr_{\check G}^{\ct}$, extending the work \cite{BMM,G1,YZ} to the twisted setting, and the study of the nearby cycle functor along a degeneration of $\Gr_{\check G}$ to $\Gr_{\check G}^{\ct}$ introduced in \cite{Z}.
The proof strategy here is inspired by our previous work \cite{CMNO} on quaterninoic Satake equivalence.\footnote{In \emph{loc. cit.}, we study a degeneration of the affine Grassmanian 
$\Gr$ for $GL_{2n}$ to a real affine Grassmannian $\Gr_\bbR$
for the quaternionic linear group
$GL_n(\mathbb H)$
and uses it to  
establish the derived Satake equivalence for $\Gr_\bbR$.}
The proof of~\eqref{intro:twisted Ginzburg} uses  
the  observation  that the natural embedding 
$\check\rho:\Gr_{\check G}^{\ct}\to\Gr_{\check G}$ satisfies the 
assumptions introduced in \cite[Section 1.1]{G2} (see Lemma \ref{S1-S3})

The methods in the paper are applicable to 
a more general setting.
 In Section \ref{variant}, we discuss the cases 
of the triality of $Spin_8$, the  special  parahoric of 
$SU_{2n+1}(F')$ coming from an order $4$ automorphism of $SL_{2n+1}$ with fixed points subgroup isomorphic to $Sp_{2n}$
(see, e.g., \cite[Section 2.1]{BH}),
and
the subgroup $\check G(F_n)\subset \check G(F)$ where 
$F_n=\bC((t^n))$.
In particular, the above first two cases 
plus the list of in Example~\ref{list 2} 
cover all the 
twisted affine Grassmannians considered in \cite{Z} (see Remark \ref{complete}).

Finally,
   similar to the derived Satake equivalence for complex reductive groups \cite{BF}, the  equivalence~\eqref{intro: equ main 2}
    has a quantum counterpart for the loop rotation equivariant sheaves on $\check X_{\ct}(F)$ where the right hand side becomes the category of (graded) $(\fg,K)$-modules. 
  The details will appear in a forthcoming  work.

\subsection{Examples}
Let $G$ be complex adjoint semisimple group and 
$\check G$ be its complex Langlands dual group.
We have the following symmetric subgroups 
$K=G^{\theta}$, $H=G^{\theta_0}$,
$\check H=\check G^{\ct_0}$, and twisted loop group $\check G(F)^{\ct}$.
Here  $\theta_0$ 
and $\ct_0$
are pinned involutions of $G$ and $\check G$ with the same image in the 
outer automorphism group 
$\on{Out}(G)\cong\on{Out}(\check G)$,
$\theta$ is a quasi-split involution 
of $G$
inner to $\on{Chev}\circ\theta_0$, 
and $\check\theta:\check G(F)\to \check G(F)$, $\ct(\gamma)(t)=\ct_0(\gamma(-t))$.
Here is the list of all possible $(G,K,H,\check G,\check H,\check G(F)^{\ct})$:
\begin{enumerate}
    \item Assume $(G,K)$ is a split symmetric pair. Then $\theta_0=\check\theta_0=\on{id}$ and we have 
    \[(G,K,H,\check G,\check H,\check G(F)^{\ct})=(G,K,G,\check G,\check G,\check G(F'))\]
    
    \item
     Assume $(G,K)$ is a quasi-split but non split symmetric pair.
Then $\theta_0\neq\on{id}$ and  $G$ must be of type 
$A_n, n\geq 2, D_n, E_6$. We have 
\[(PGL_{2n+1},P(GL_{n+1}\times GL_n), SO_{2n+1},SL_{2n+1},SO_{2n+1}, SU_{2n+1}(F'))\]
\[(PGL_{2n},P(GL_n\times GL_n)\rtimes\mu_2,PSp_{2n},SL_{2n},Sp_{2n},SU_{2n}(F')),\ n\geq 2\]
\[(PSO_{2n},P(O_{n-1}\times O_{n+1}),SO_{2n-1},Spin_{2n},Spin_{2n-1}, Spin_{2n}^{(2)}(F'))\]
\[
(E_6^{ad},SL_6\times SL_2/\on{center},F_4,E_6^{sc},F_4, E^{sc,(2)}_6(F'))\]
\end{enumerate}
Here 
$SU_{2n}(F'),SU_{2n+1}(F')$ are the quasi-split special unitary groups,
$Spin_{2n}^{(2)}(F')$ and $E_6^{sc,(2)}(F')$ are the quasi-split but non-split forms of 
$Spin_{2n}(F')$ and $E^{sc}_6(F')$.
\subsubsection{Galois periods}
Consider the  split symmetric space 
$X=PGL_{2n}/PO_{2n}$ for $G=PGL_{2n}$.
The twisted affine Grassmannian $\Gr_{SL_{2n}}^{\ct}=SL_{2n}(F')/SL_{2n}(\calO')\cong SL_{2n}(F)/SL_{2n}(\calO)$
is isomorphic to the  affine Grassmannian
for $SL_{2n}$
and the loop symmetric space is given by $\check X_{\ct}(F)=SL_{2n}(F)/SL_{2n}(F')$. 
Theorem \ref{intro:main 3}
reduces to the usual derived Satake equivalence for $SL_{2n}$:
 \[D^b_{SL_{2n}(\calO')}(SL_{2n}(F')/SL_{2n}(\calO'))\cong \on{Perf}^{PGL_{2n}\times PGL_{2n}}(T^*PGL_{2n}[2])\]
 and Theorem \ref{intro:main 2} says:
\[D^b_{SL_{2n}(\calO)}(SL_{2n}(F)/SL_{2n}(F'))_e\cong \on{Perf}^{PGL_{2n}}(T^*(PGL_{2n}/PO_{2n})[2]).\]

\subsubsection{Unitary periods}
Consider the  quasi-split symmetric space 
$X=PGL_{2n}/P(GL_n\times GL_n)\rtimes\mu_2$ for $G=PGL_{2n}$, $n
\geq 2$.
The twisted affine Grassmannian is given by 
$\Gr^{\ct}_{SL_{2n}}=SU_{2n}(F')/ SU_{2n}(\calO')$
and the loop symmetric space  $\check X_{\ct}(F)=SL_{2n}(F)/SU_{2n}(F')$ is the space of Hermitian matrices over $F$ with determinant one.
Theorem \ref{intro:main 3} and Theorem \ref{intro:main 2} say:
\[D^b_{SU_{2n}(\calO')}(SU_{2n}(F')/ SU_{2n}(\calO'))\cong \on{Perf}^{PSp_{2n}\times PSp_{2n} }(T^*PSp_{2n}[2])\]
\[D^b_{SL_{2n}(\calO)}(SL_{2n}(F)/SU_{2n}(F'))_e\cong \on{Perf}^{PGL_{2n}}(T^*(PGL_{2n}/P(GL_n\times GL_n)\rtimes\mu_2)[2]).\]

\subsection{Tamely ramified equivalence}
Using the recent work \cite{LPY}, we deduce from 
Theorem \ref{intro:main 2}
a version of Bezrukavnikov equivalence \cite{B}
for quasi-split symmetric spaces.
Namely,
consider the base change 
$T^*X\times_{\fg^*}\widetilde\fg$ 
of the Grothendieck-Springer resolution
$\widetilde\fg\to\fg^*$
along the moment map $T^*X\to\fg^*$.
Let $\check I\subset\check G(\calO)$ be an Iwahori subgroup and 
let 
$D^b_{\check I}(\check X_{\ct}(F))_e\subset D^b_{\check I}(\check X_{\ct}(F))$ 
the full subcategory 
of the $\check I$-equivariant derived category of $\check X_{\ct}(F)$
generated by $\omega_{e}$, viewed as an object in $D^b_{\check I}(\check X_{\ct}(F))$, under the
affine Hecke action.
We have the following 
constructible realization of 
$G$-equivariant perfect complexes on 
$T^*X\times_{\fg^*}\widetilde\fg$:
\begin{thm}\label{Intro: main 3}
 There is a natural equivalence
 \[D^b_{\check I}(\check X_{\ct}(F))_e\cong 
 \on{Perf}^G((T^*X\times_{\fg^*}\widetilde\fg)[2]) \]
 \end{thm}
Theorem \ref{Intro: main 3}
is restated in Theorem \ref{Bez equivalence}.
We refer to Section \ref{Bez equivalence} for a more
detailed explanation of the statement including its 
relationship with \cite{B}. 
\subsection{Twistor $\mathbb P^1$}\label{twistor}
We were led to the derived Satake equivalence for the loop symmetric spaces~\eqref{intro: equ main 2} from the connections to 
geometric Langlands on the twistor $\mathbb P^1$ pioneered in \cite[Section 4.9]{BZN}.
We briefly explain the connections here,
including a Twisted Relative Langlands Duality conjecture providing a spectral descripton of the whole constructible category of 
$\check X_{\ct}(F)$.
Details will  appear in  forthcoming works
\cite{C,CY3,CNY}.

Let $\check\eta_0$ be the conjugation 
on $\check G$ such that such that 
$\check\eta_0\circ\check\theta_0=\check\theta_0\circ\check\eta_0$ is a compact conjugation
and let $\check G_\bbR=\check G^{\check\eta_0}$
be the corresponding real group.
Let $\Bun_{\check G}(\mathbb P^1)$ the moduli stack of $\check G$-bundles on the complex projective line  $\bbP^1$.
The Galois group $\on{Gal}(\bbC/\bbR)$ of $\bbR$
acts on $\bbP^1$ via the 
the antipodal conjugation 
$\alpha:\mathbb P^1\to\mathbb P^1$ sending $\alpha(z)=-\bar z^{-1}$,
 and we denote by 
$\widetilde{\mathbb P}^1_\bbR=\mathbb P^1/\on{Gal}(\bbC/\bbR)$ 
the Galois-descent known as the
twistor $\mathbb P^1$.
The conjugation $\check\eta_0$ on $\check G$ and $\alpha$ on $\bbP^1$ induces a conjugation 
$\check\eta:\Bun_{\check G}(\bbP^1)\to \Bun_{\check G}(\bbP^1)$ and we denote by
$\Bun_{\check G_\bbR}(\widetilde{\mathbb P}^1_\bbR)=
\Bun_{\check G}(\bbP^1)^{\check\eta}$ the real analytic stack of real points of 
$\Bun_{\check G}(\bbP^1)$. It classifies a $\check G$-bundle $\mE$ on $\bbP^1$
together with an isomorphism $\iota:\mE\is \check\eta(\mE)$
such that the composition
\[\mE\stackrel{\iota}\longrightarrow \check\eta(\mE)\stackrel{\check\eta(\iota)}\longrightarrow \check\eta(\check\eta(\mE))\]
is the identity.
In the forthcoming works \cite{CY3,CNY}, 
we establish a Matsuki equivalence
\begin{equation}\label{intro:Matsuki}
  D^b_{\check G(\calO)}(\check X_{\ct}(F))\cong D_!(\Bun_{\check G_\bbR}(\widetilde{\mathbb P}^1_\bbR))  
\end{equation}
relating the  $\check G(\calO)$-equivariant derived 
category of 
$\check X_{\ct}(F)$ with the 
 category of constructible complexes  on the moduli stack  
real bundles on $\widetilde{\mathbb P}^1_\bbR$ that are extension by zero off of finite type substacks.
The equivalence~\eqref{intro:Matsuki}
 is a loop symmetric space version of the Matsuki equivalence in
\cite[Theorem 1.2]{CN1} relating the  
the  $\check G(\calO)$-equivariant derived 
category of loop space $\check X(F)$
of the symmetric space $\check X=\check G/\check G^{\ct_0}$ (instead of the loop symmetric space $\check X_{\ct}(F)$ considered in the paper) with the 
 category of constructible complexes  on the moduli stack  
real bundles on the real projective line ${\mathbb P}^1(\mathbb R)$.

On the other hand, consider
 the $\theta$-anti-fixed point 
$X_\theta=G^{\on{inv}\circ\theta}=\{g\in G|\theta(g)=g^{-1}\}$, viewed as a $G$-variety with the $\theta$-twisted conjugation action, and 
the semi-direct product
$G^L=G\rtimes_{\theta}\mu_2$, where  
$\mu_2$ acts on $G$ via the involution
$\theta$.
Denote by
$\on{Loc}_{G^L}(\widetilde{\mathbb P}^1_\bbR)$ the derived Betti moduli of  
$G^L$-local system on $\widetilde{\mathbb P}^1_\bbR$ together with an isomorphism 
of the induced $\mu_2$-local system on $\widetilde{\mathbb P}^1_\bbR$ with the 
one given by the $\mu_2$-covering $\mathbb P^1\to\widetilde{\mathbb P}^1_\bbR$.
Then there is Koszul duality equivalence
\begin{equation}\label{kd}
    \on{Perf}^G(T^*X_\theta[2])\cong \on{Coh}(\on{Loc}_{G^L}(\widetilde{\mathbb P}^1_\bbR)) 
\end{equation}
relating the dg derived category 
of coherent complexes on the cotangent bundle $T^*X_\theta$ and $\on{Loc}_{G^L}(\widetilde{\mathbb P}^1_\bbR)$. We now can state:
\begin{conjecture}\label{GL for twistor}
There are horizontal equivalences 
in the following  diagram 
\[\xymatrix{D^b_{\check G(\calO)}(\check X_{\ct}(F))\ar[r]^{\Upsilon_{\ct}\ \ }_{\simeq\ \ \ }\ar[d]_{\simeq}^{~\eqref{intro:Matsuki}}&\on{Perf}^G(T^*X_\theta[2])\ar[d]_{\simeq}^{\eqref{kd}}\\
D_!(\Bun_{\check G_\bbR}(\widetilde{\mathbb P}^1_\bbR))\ar[r]^{\Upsilon_\bbR}_{\simeq}&\on{Coh}(\Loc_{G^L}(\widetilde{\mathbb P}^1_\bbR))}\]
\end{conjecture}

We will refer the conjectural equivalence 
$\Upsilon_{\ct}$ (resp. $\Upsilon_\bbR$) the twisted relative Langlands duality (resp. geometric Langlands on the twistor $\mathbb P^1$).
The commutative diagram above basically says that, under the Matsuki equivalence in~\eqref{intro:Matsuki},
the two conjectures are  \emph{Koszul dual} to each other.

The derived Satake equivalence 
for loop symmetric spaces 
in Theorem \ref{intro:main 2}
provides an evidence for Conjecture \ref{GL for twistor}.
Namely, according to \cite[Section 9]{R}, 
$X_\theta$ is a finite union of 
$\theta$-twisted closed $G$-orbits 
\[X_\theta=\bigsqcup_{\sigma\in \Sigma} X_\sigma\cong \bigsqcup_{\sigma\in \Sigma} G/K_\sigma\]
Here $\Sigma$ is the set  of conjugacy classes of involutions which are inner
to $\theta$, and 
$K_\sigma=G^{\sigma}$
is the symmetric subgroup 
associated to the involution 
$\sigma$.
The base point $\theta\in \Sigma$
corresponds to the symmetric space 
$X=G/K$ and hence an embedding 
$\on{Perf}^G(T^*X[2])\hookrightarrow\on{Perf}^G(T^*X_\theta[2])\cong\bigoplus_{\sigma\in\Sigma}\on{Perf}^G(T^*X_\sigma[2])$.
We expect Conjecture \ref{GL for twistor} is compatible with 
the equivalence 
$D^b_{\check G(\calO)}(\check X_{\ct}(F))_{e}\cong 
 \on{Perf}^G(T^*X[2])$ in 
Theorem \ref{intro:main 2}
in the sense that 
the following  diagram is commutative:
\[\xymatrix{D^b_{\check G(\calO)}(\check X_{\ct}(F))_e\ar[r]_{\simeq}^{\eqref{intro: equ main 2}}\ar@{^{(}->}[d]_{}^{}&\on{Perf}^G(T^*X[2])\ar@{^{(}->}[d]_{}^{}\\
D^b_{\check G(\calO)}(\check X_{\ct}(F))\ar[r]^{\Upsilon_{\ct}}_{\simeq}&\on{Perf}^G(T^*X_\theta[2])}\]
where the vertical arrows are the natural full embeddings.

\begin{remark}
(1) 
There is also  a tamely ramified version
of Conjecture \ref{GL for twistor}
relating the 
tamely ramified and analytic
geometric Langlands on the twistor $\mathbb P^1$
in \cite[Section 3.2]{BZN} and \cite{S}
with the Iwahori equivariant derived category of $\check X_{\ct}(F)$, see \cite{C}. 
Theorem \ref{Intro: main 3} provides an evidence of such a generalization.

(2) In \cite{Da}, Dougal Davis independently comes up with a similar 
equivalence $\Upsilon_{\ct}$ 
from a different consideration.
\end{remark}

\subsection{Organization}
In Section \ref{KR theory}, we recall 
the Kostant-Rallis theory and provide a description of the cotangent bundle of 
quasi-split symmetric spaces in terms of regular centralizers group schemes.
In Section \ref{twisted Satake}, 
we prove the derived Satake equivalence for twisted affine Grassmannians.
In Section \ref{TRS}, 
we study ring objects in the twisted derived Satake category arsing from loop symmetric spaces. 
We generalize the results \cite{G2}
and apply them to prove 
Theorem \ref{intro:main 1} and Theorem \ref{intro:main 2}.

\subsection{Acknowledgements} 
The paper is inspired by the 
work of Ben-Zvi and Nadler \cite{BZN}
on geometric Langlands on the twistor $\mathbb P^1$ and the work of 
Ginzburg \cite{G2}.
The author would like to thank 
Sanath Devalapurkar,
 Dougal Davis, 
 Victor Ginzburg,
Jiuzu Hong, Mark Macerato,
 David Nadler, Akshay Venkatesh, 
Weiqiang Wang 
 and Lingfei Yi for many useful discussions.
The research of
T.-H.~Chen is supported by NSF grant DMS-2143722.

\section{Cotangent bundles of symmetric spaces}\label{KR theory}
\subsection{Symmetric spaces}\label{symemtric pair}
Let $G$ be a complex adjoint semisimple group.
We fix a pinning $(G,B,T,x)$ where $(B,T)$ is a Borel pair 
and $x=\sum_{\alpha\in\Delta} x_\alpha\in\on{Lie}B$,
here $\Delta$ is the set of simple roots and 
$x_\alpha$ is a basis for the weight space $(\on{Lie}G)_\alpha$. 

Let $\on{Aut}(G,B,T,x)\subset\on{Aut}(G)$ be the subgroup of pinned automorphisms. Since $G$ is adjoint there is an isomorphism between $\on{Aut}(G,B,T,x)$ with the automorphism group of the Dynkin diagram of $G$.
Let $\on{Chev}\in \on{Aut}(G,B,T,x)$ be the Chevalley involution, that is, the unique involution 
such that $\on{Chev}(t)=w_0(t^{-1})$
for any $t\in T_0$, here $w_0$ is the longest element in the Weyl group $\rW_G=N_G(T)/T$.
We have $\on{Chev}=\on{id}$ if and only if 
$G$ is not of type $A_n, D_{2n+1}, E_6$. 
According to \cite[Lemma 5.3 and 5.4]{AV}, we can find a
lifting $n_0\in N_G(T)$ of $w_0$
such that $n_0^2=1$
and $\sigma(n_0)=n_0$ for all $\sigma\in\on{Aut}(G,B,T,x)$.
In particular, the composition 
$\on{Ad}_{n_0}\circ\on{Chev}$ is an involution satisfying 
$\on{Ad}_{n_0}\circ\on{Chev}(t)=t^{-1}$
for all $t\in T_0$, that is, a split involution.

Let $\theta_0\in\on{Aut}(G,B,T,x)$
be a pinned involution.
Then the composition
\begin{equation}\label{theta}
\theta=\on{Ad}_{n_0}\circ\on{Chev}\circ\theta_0
\end{equation}
defines a quasi-split involution. Indeed we have 
$\theta(B)\cap B=n_0Bn_0^{-1}\cap B=T$.
We will write 
$H=G^{\theta_0}$
and $K=G^{\theta}$
for the corresponding symmetric subgroups.

\begin{example}\label{list}
(1) If $\theta_0=\on{id}$, we have $H=G$
and $(G,K)$ is a split symmetric pair.

(2) If $\theta_0\neq\on{id}$, then $G$ must be of type 
$A_n, n\geq 2, D_n, E_6$ and $(G,K)$ is a quasi-split but not split 
symmetric pair. Here is a list of the triple
$(G,K,H)$:
\[(PGL_{2n+1},P(GL_{n+1}\times GL_n), SO_{2n+1}), (PGL_{2n},P(GL_{n}\times GL_n)\rtimes\mu_2, PSp_{2n}), n\geq 2,\]
\[(PSO_{2n},P(O_{n-1}\times O_{n+1}),SO_{2n-1}),
(E_6^{ad},SL_6\times SL_2/\on{center},F_4).\]
\end{example}

Although $K$ might be disconnected,  
from the list above, we see that 
$H=G^{\theta_0}$ 
and $T_H=T^{\theta_0}$ are connected \cite[Lemma 4.6]{Z}.
Since 
$(T,B)$ is a preserved by $\theta_0$, 
it restricts to a Borel pair 
$(T_H=T\cap H,B_H=B\cap H)$ of $H$.
We have  
$\theta\circ\theta_0=\theta_0\circ\theta$
and $\theta$ restricts to a split involution on $H$ with 
symmetric subgroup $M=H^\theta$ and 
split maximal torus 
$A=T_H\subset H$.
Indeed, 
we have $\theta(t)=\on{Ad}_{n_0}\circ\on{Chev}(a)=a^{-1}$
for $a\in A$.

We denote by $X=G/K$ and $Y=H/M$ the symmetric spaces associated to the 
symmetric pairs $(G,K)$ and $(H,M)$.

\subsection{Reminder on Kostant-Rallis theory}
Let $\fg,\fb,\ft,\fk,\frak h,\frak m,\fa=\ft_H$ be the Lie algebras of 
$G,B,T,K,H,M, A=T_H$ and $\fg^*,\fb^*,\ft^*,\fk^*,\frak h^*,\fa^*=\ft_H^*$ be their dual spaces.  
Consider the  Cartan decompositions
$\fg=\fk\oplus\fp$ and $\frak h=\frak m\oplus\frak q$
where $\fp$ is the $(-1)$-eigenspace of 
$\theta$ and $\frak q=\frak h\cap\fp$. 
Note that $\fa=\ft_H$ is at the same time 
the maximal abelain subalgebra $\fa\subset\fp$ of $\fp$
and the split maximal Cartan subalgebra of $\ft_H\subset \frak h$ of $\frak h$.
Note that the Cartan-Killing form on $\fg$ 
is $\theta$ (resp. $\theta_0$)
invariant and we will use it  to identify $\fg\cong\fg^*$, $\fk\cong\fk^*$, $\fh\cong\fh^*$, etc.

Let $\rW_H=N_H(\fa)/Z_H(\fa)$ be the Weyl group of $H$.
Let $\rW_X=N_K(\fa)/Z_K(\fa)$ 
and $\rW_Y=N_M(\fa)/Z_M(\fa)$
be the little Weyl groups of the symmetric pair 
$(G,K)$ and $(H,M)$.
Denote by $\fc=\ft^*//\rW$, $\fc_H=\fa^*//\rW_H$,
and $\fc_X=\fa^*//\rW_X$, $\fc_Y=\fa^*//\rW_Y$.
Since $(H,M)$ is a split pair we have 
$\rW_H=\rW_Y$ and $\fc_H=\fc_Y$.

\begin{lemma}\label{equ}
(1) We have $\theta=-\theta_0$ on $\ft$ and hence an isomorphism $\rW_{G}^\theta\cong\rW_G^{\theta_0}$.

(2) We have the following diagram of isomorphisms
\begin{equation}\label{Weyl groups}
\xymatrix{\rW_H\ar[r]^\simeq\ar[d]^\simeq&\rW^{\theta_0}\ar[d]^\simeq\\
    \rW_\fa\ar[r]^\simeq&\rW^{\theta}}
\end{equation}
where the horizontal arrows  are the natural maps and the left vertical arrow is induced by the natural  map 
    $\rW_H=N_M(\fa)/Z_M(\fa)\to N_K(\fa)/Z_K(\fa)=\rW_\fa$.
    
    (3) The natural embedding 
    $\fa^*\to\ft^*$ induces a closed embedding $\fc_X\cong\fc_H\to\fc$.
\end{lemma}
\begin{proof}
(1) This follows from the formula~\eqref{theta}.

(2)   
 Since $\theta$ is quasi-split, the isomorphism $\rW_\fa\cong\rW^\theta$
well-known, see for example, \cite[Lemma 1.6]{L}. 
The isomorphism $\rW_H\cong\rW_\fa$
is obvious if $\theta_0=\on{id}$ since  $M=K$.
    If $\theta$ is non-trivial, then one
    verify the claim from the list in Example \ref{list}.
    Now the 
    identification $\rW_H\cong\rW^{\theta_0}$
    follows from (1) and the diagram~\eqref{Weyl groups}

(3) This is proved in \cite[Lemma 2.27]{HM}.
\end{proof}

Consider the Cartan decompositions 
 $\fg^*=\fk^*\oplus\fp^*$
and $\frak h^*=\frak m^*\oplus\frak q^*$.
We have embedding $\fa^*\subset\frak q^*\subset\fp^*$.
Let $\fg^{*,r}\subset\fg^*$ and $\fh^{*,r}\subset\fh^*$
be the open subset of 
regular elements.
Let $\fp^{*,r}\subset\fp^*$ be the open subset of $K$-regular elements, that is,
those elements $x\in\fp^*$
such that the co-adjoint orbit $K\cdot x$ has maximal dimension.
Similarly, we set $\frak q^{*,r}$ the set of $M$-regular elements in $\frak q^*$.
Let $\calN_{\fg^*}\subset\fg^*$ be the nilpotent element  in $\fg^*$. We write
$\calN_{\fp^*}=\fp^*\cap\calN$
and $\calN_{\fp^{*,r}}=\fp^{*,r}\cap\calN$.

\begin{lemma}\label{regular}
(1) There is an isomorphism $\fp^*//K\cong\fc_X$.

(2) Every semisimple element in $\fp^{*}$ (resp. $\frak q^*$) is $K$-conjugate (resp. $M$-conjugate) to an element in $\fa^*$.

(3) The co-adjoint action of 
$K$ on $\calN_{\fp^{*,r}}$ is transitive.

(4) We have $\frak q^{*,r}=\frak q^*\cap\frak p^{*,r}=\fq^*\cap\fh^{*,r}$.

(5) The set of irregular elements 
$\fp^*\setminus\fp^{*,r}$ is  of codimension $\geq2$.

(6) We have $\fh^{*,r}=\fh^*\cap\fg^{*,r}$,
$\fp^{*,r}=\fp^*\cap\fg^{*,r}$
and  $\frak q^{*,r}=\frak q^*\cap\fg^{*,r}$.
In particular, there is a 
Cartesian diagram \[\xymatrix{\fq^{*,r}\ar[r]\ar[d]&\fp^{*,r}\ar[d]\\
\fh^{*,r}\ar[r]&\fg^{*,r}}\]
\end{lemma}
\begin{proof}
    Part (1)-(4) are proved in \cite{KR}. Part (5)  is proved in \cite[Lemma 6.31]{Le}
    and part (6) follows from (4) and 
   \cite[proposition 2.24]{HM}.
\end{proof}

\begin{remark}
    Part (1)-(5) are true for all symmetric pairs but 
    part (6) is only true for quasi-split symmetric pairs.
\end{remark}

Consider the 
Chevalley maps 
$\chi:\frak g^*\to\frak g^*//G\cong\fc$,
$\chi_H:\frak h^*\to\frak h^*//H\cong\fc_H$
and 
$\chi_X:\fp^*\to\fp^*//K\cong\fc_X$,
 $\chi_Y:\frak q^*\to\frak q^*//M\cong\fc_Y$.

\begin{lemma}\label{section}
    There exists a $\mathfrak{sl}_2$-triple $\{h,e,f\}\subset\frak h\cong\fh^*$, $[h,e]=2e, [h,f]=-2f, [e,f]=h$, 
    satisfying the following:
    
    (1) We have $h\in\fm^*$,
    $e,f\in\frak q^*$
    
    (2) We have $f+Z_{\fg^*}(e)\subset\fg^{*,r}$
    
    (3) We have $f+Z_{\fp^*}(e)=f+Z_{\frak q^*}(e)=f+Z_{\frak h^*}(e)\subset\frak q^{*,r}$
    
    (4) We have a commutative diagram of maps
    \begin{equation}
        \xymatrix{
        f+Z_{\frak h^*}(e)\ar[r]&\frak h^*\ar[r]^{\chi_H}&\fc_H\\
        f+Z_{\frak q^*}(e)\ar[r]\ar[d]^\simeq\ar[u]_\simeq&\frak q^*\ar[r]^{\chi_Y}\ar[d]\ar[u]&\fc_Y\ar[u]_\simeq\ar[d]^\simeq\\
        f+Z_{\frak p^*}(e)\ar[r]\ar[d]&\frak p^*\ar[r]^{\chi_X}\ar[d]&\fc_X\ar[d]\\
        f+Z_{\fg^*}(e)\ar[r]&\fg^*\ar[r]^\chi&\fc}
    \end{equation}
    where the non-label arrows are closed embeddings 
    and  the horizontal composted maps induce isomorphisms
\begin{equation}\label{KR sections}
    f+Z_{\frak g^*}(e)\cong\fc,
f+Z_{\frak h^*}(e)\cong\fc_H, 
f+Z_{\frak q^*}(e)\cong\fc_Y, f+Z_{\frak p^*}(e)\cong\fc_X.
\end{equation}

(5) Let $\phi:SL_2\to H$ be the map corresponding to the 
$\frak{sl}_2$-triple.  The space $f+Z_{\fg^*}(e)$ (resp. $f+Z_{\frak h^*}(e)= 
f+Z_{\frak q^*}(e)=f+Z_{\frak p^*}(e)$) is stable under the $\bG_m$-action 
$v\to t^2\on{Ad}_{\phi(t)}(v)$ and the 
isomorphisms in~\eqref{KR sections} are $\bG_m$-equivariant where 
$\bG_m$ acts on $\fc$ (resp. $\fc_X=\fc_Y=\fc_H$) be the 
square of the natural $\bG_m$-action.

   \end{lemma}
\begin{proof}
Pick $e\in\calN_{\fp^{*,r}}$.
Lemma \ref{regular} (6)
implies $e\in\fg^{*,r}$ and (1) and (2)
follow from
\cite[Proposition 4]{KR} and \cite{Kos}.
Note that the  roots  of $\frak h$ with respect to 
$\fa$ are given by those restricted roots 
$\alpha$ of 
$\fg$ with respect to $\fa$ such that $\alpha/2$ is not a root (the only non-trivial case is when $\theta_0\neq\on{id}$
and 
one can verify this using the list in Example \ref{list}.) It follows that $\frak h\subset\fg$ is equal to the 
semisimple subalgebra $\tilde\fg\subset\fg$ in \cite[Proposition 23]{KR}
and part (3) and (4) follows from
\cite[Theorem 11,12,13]{KR} and \cite{Kos}.

\end{proof}

\subsection{Regular centralizers}\label{involutions}
The isomorphism in Lemma \ref{section} (4) gives rise to sections of the Chevalley maps
$\kappa:\fc\to\fg^*$, $\kappa_X:\fc_X\to\fp^*$,
$\kappa_H:\fc_H\to\frak h^*$, known as the Kostant and 
Kostant-Rallis sections and the $\frak{sl}_2$-triple $\{h,e,f\}$
is called the principal normal $\frak{sl}_2$-triple. 

Let $I_G\to\fg^*$ (resp. $I_H\to\frak h^*$) be the group scheme of centralizers.
Let $I_X\to\fp^*$ be the group scheme of $K$-centralizers.

\begin{lemma}\label{J}
    (1)
   The restriction $I_G|_{\fg^{*,r}}$
   descends to a 
    smooth commutative group scheme
     $J_G\to\fc$
     along the map $\chi|_{\fg^{*,r}}:\fg^{*,r}\to\fc$.
    Moreover, the canonical isomorphism 
    $\chi^*J_G|_{\fg^{*,r}}\cong I_G|_{\fg^{*,r}}$ admits a canonical extension 
   $\chi^*J_G\to I_G$.
    The same is true for $I_H\to\fh^*$ and we denote by 
    $J_H\to\fc_H$ the corresponding group scheme.
    
    (2) The restriction 
    $I_X|_{\fp^{*,r}}$ descends to a smooth commutative group scheme
     $J_X\to\fc_X$  along the map $\chi_X|_{\fp^{*,r}}:\fp^{*,r}\to\fc_X$. Moreover, 
     the canonical isomorphism 
     $\chi_X^*J_X|_{\fp^{*,r}}\cong I_X|_{\fp^{*,r}}$ admits a canonical extension
    $\chi_X^*J_X\to I_X$. 

    (3) The Kostant section 
    $\kappa:\fc\to\fg^*$ 
   induces an isomorphism 
   $J\cong\kappa^*I_G$. 
   Consider the $\bG_m$-actin on 
   $G\times\fc$ (resp. $\fg^*$)
   given by $t\cdot(g,c)=(g\phi(t),t^{-2}c)$ (resp. $t\cdot v=t^{-2}v$).
   Here $\phi$ is the $SL_2$-triple in Lemma \ref{section} (5).
   Then we have a $\bG_m$-equivariant map
   $G\times\fc\to\fg^{*}$ sending $(g,v)\to\on{Ad}_g(\kappa(v))$. Moreover, it   
   factors through a $\bG_m$-equivaraint isomorphism $(G\times\fc)/J\cong\fg^{*,r}$,
    where $J_G$ acts on $G\times\fc$ via the embedding
    $J\cong\kappa^*I_G\subset G\times\fc$
    
    (4) 
    The Kostant-Rallis section 
    $\kappa_X:\fc_X\to\fp^*$ 
   induces an isomorphism 
   $J_X\cong\kappa_X^*I_X$.
   Consider the $\bG_m$-actin on 
   $K\times\fc_X$ (resp. $\fp^*$)
   given by $t\cdot(k,c)=(k\phi(t),t^{-2}c)$ (resp. $t\cdot v=t^{-2}v$).
  Then we have a $\bG_m$-equivariant map
 $K\times\fc_X\to\fp^{*}$ sending $(k,c)\to\on{Ad}_x(\kappa_X(v))$. Moreover, it factors through an isomorphism $(K\times\fc_X)/J_X\cong\fp^{*,r}$,
 where $J_X$ acts on $K\times\fc_X$ via the embedding
    $J_X\cong\kappa_X^*I_X\subset K\times\fc_X$

\end{lemma}
\begin{proof}
  Part  (1) is proved in \cite{N} and 
 Part  (2) is proved in \cite[Proposition 4.1]{HM}.
 Part (3) and (4) follows from \cite[Theorem 9]{KR}, \cite{Kos}, and Lemma \ref{section} (5).
\end{proof}
\begin{remark}\label{single orbit}
    The assumption that $G$ is adjoint is essential for part (4). More precisely, the statement holds true as lone as 
    $K$-acts transitively on $\calN_{\fp^{*,r}}$. This is true if $G$ adjoint group but fails in geneal, for example, for the pair $(SL_2,SO_2)$, there are two $SO_2$-orbits in 
    $\calN_{\fp^{*,r}}$.
\end{remark}

Consider the base change $I_G|_{\fh^{*}}$
along the embedding $\fh^*\subset\fg^*$. 
We have a natural involution 
\[\theta_{0,I}:I_G|_{\fh^{*}}\to I_G|_{\fh^{*}}\ \ \theta_{0,I}(g,v)=(\theta_0(g),v)\]
with fixed points subgroup scheme
$(I_G|_{\fh^{*}})^{\theta_0}=I_H$.
It follows from 
Lemma \ref{regular} and Lemma \ref{J} that 
the restriction of above involution to the regular locus $\fh^{*,r}$
descends 
an involution 
\[\theta_{0,J}:J_G|_{\fc_H}\to J_G|_{\fc_H}\] with fixed points subgroup scheme
$(J_G|_{\fc_H})^{\theta_{0,J}}\cong J_H$.
Consider the norm map
\begin{equation}\label{norm}
    \on{Nm}:J_G|_{\fc_H}\to J_H\cong (J_G|_{\fc_H})^{\theta_{0,J}},\ \ \ \ \on{Nm}(g)=g\theta_{0,J}(g).
\end{equation}
and the anti-involution
\[\on{inv}\circ\theta_{0,J}:J_G|_{\fc_H}\to J_G|_{\fc_H}\ \ g\to \theta_{0,J}(g)^{-1}\]

\begin{lemma}\label{exact}
There is a short exact sequence of smooth affine group schemes
\[0\to(J_G|_{\fc_H})^{\on{inv}\circ\theta_{0,J}}\to J_G|_{\fc_H}\stackrel{\on{Nm}}\to J_H\to0.\]

  \end{lemma}
\begin{proof}
The smoothness of $(J_G|_{\fc_H})^{\on{inv}\circ\theta_{0,J}}$ follows from \cite[Corollary 9.8]{Gr}.
    It is clear that $(J_G|_{\fc_H})^{\on{inv}\circ\theta_{0,J}}$ is the kernel of 
    $\on{Nm}$ thus it remains to show that $\on{Nm}$ is surjective.
    For this we observe that $H$ is adjoint and hence $J_H$ has connected fibers 
    and the desired claim follows from \cite[Proposition 9.1]{R}.
\end{proof}

\begin{example}
Assume $\theta_0=\on{id}$. Then 
$G=H$ and 
the norm map is the square map $\on{Nm}=[2]:J_G\to J_G, g\to g^2$, 
$J_G^{\on{inv}\circ\theta_0}=J_G[2]$ is the order two elements in $J_G$, and the exact sequence above becomes
\[0\to J_G[2]\to J_G\stackrel{[2]}\to J_G\to 0.\]
\end{example}

We shall relate the two group schemes $J_X$ and the kernel $(J_G|_{\fc_H})^{\on{inv}\circ\theta_{0,J}}=\on{Ker}(\on{Nm}:J_G|_{\fc_H}\to H_H)$.
Consider the base change $I_G|_{\fp^{*}}$
along the embedding $\fp^*\subset\fg^*$. 
We have a natural involution 
\[\theta_{I}:I_G|_{\fp^{*}}\to I_G|_{\fp^{*}}\ \ \theta_{I}(g,v)=(\theta(g),v)\]
with fixed points subgroup scheme
$(I_G|_{\fp^{*}})^{\theta_I}=I_X$.
It follows from Lemma \ref{regular} and Lemma \ref{J} that 
the restriction of above involution to the regular locus $\fp^{*,r}$
descends 
an involution 
\[\theta_{J}:J_G|_{\fc_X}\to J_G|_{\fc_X}\] with fixed points subgroup scheme
$(J_G|_{\fc_X})^{\theta_{J}}\cong J_X$.
We have the following key observation
relating the two subgroup schemes
\begin{equation}\label{embeddings}
   J_X\cong J_G|_{\fc_X}^{\theta_J}\longrightarrow J_G|_{\fc_X=\fc_H}\longleftarrow  (J_G|_{\fc_H})^{\on{inv}\circ\theta_{0,J}}
\end{equation}

\begin{lemma}\label{exact}
The two subgroup schemes $J_X$ and $(J_G|_{\fc_H})^{\on{inv}\circ\theta_{0,J}}$ of  
$J_G|_{\fc_X=\fc_H}$  are equal 
  \[J_X|_{\fc_X^{rs}}=(J_G|_{\fc_H})^{\on{inv}\circ\theta_{0,J}}|_{\fc_H^{rs}}\]
  over the open subset 
  $\fc_X^{rs}=\fc_H^{rs}$.
  \end{lemma}
\begin{proof}
It follows from 
 Lemma \ref{regular} (6) and 
Lemma \ref{section} (4) that 
the restriction of the involutions $\theta_{0,I}$ and $\theta_{I}$ on
 $J_G|_{\fq^{*,r}=\fh^{*,r}\cap\fp^{*,r}}$ 
 descend to $\theta_{0,J}$ and $\theta_J$
 along the map $\chi_Y|_{\fq^{*,r}}:\fq^{*,r}\to\fc_Y=\fc_X=\fc_H$.
Thus it suffices to show that 
 $\on{inv}\circ\theta_{0,I}=\theta_{I}$ on
$I_G|_{\fq^{*,rs}}$.
By Lemma \ref{regular} (2), every element
in $\fq^{*,rs}$ is $M$-conjugate to an element
in $\fa^{*,rs}$, we reduce to check that 
 $\on{inv}\circ\theta_{0,I}=\theta_{I}$ on 
 $I_G|_{\fa^{*,rs}}$.
 Note that $I_G|_{\fa^{rs}}=T\times\fa^{*,rs}$
 and the desired equality follows from~\eqref{theta}.

\end{proof}

\subsection{Cotangent bundles of symmetric spaces as affine closures}\label{contangent bundle}
Let $T^*X=T^*(G/K)$ be the cotangent bundle of $X$ with moment map $\mu_X:T^*X\to\fg^*$.
Under the natural isomorphism 
$T^*X\cong G\times^K\fp^*$, the moment map 
is given by $\mu_X(g,v)\to \on{Ad}_g(v)$.
We shall give a description of $T^*X$ in terms of 
the regular centralizers group schemes $J_G$ and $J_H$.

Consider the quotient 
\[(G\times J_H)/J_G|_{\fc_H}=((G\times\fc_H)\times_{\fc_H} J_H)/J_G|_{\fc_H}\]
where $J_G|_{\fc_H}$ acts on $J_H$ via the norm map
$\on{Nm}:J_G|_{\fc_H}\to J_H$
~\eqref{norm}  and on $G\times\fc_H$
via the embedding $J_G\to \fc\times G$ given by the Kostant section 
$\kappa_G$
(see Lemma \ref{J} (3)).
We have a natural map 
\[\mu^\circ_G:(G\times J_H)/J_G|_{\fc_H}\to\fg^*\]
given by $(g,c,y)\to \on{Ad}_g\kappa_H(c)$.
Denote by 
\[\overline{(G\times J_H)/J_G|_{\fc_H})}=\on{Spec}(\bC[G\times J_H]^{J_G|_{\fc_H}})\]
the affine closure of $(G\times J_H)/J_G|_{\fc_H})$. The map $\mu^\circ_G$ extends to a map 
\[\mu_G:\overline{(G\times J_H)/J_G|_{\fc_H})}\to\fg^*\]

\begin{prop}\label{Wh=T^*X}
There is a horizontal isomorphism of $G$-varieties
\[\xymatrix{\overline{(G\times J_H)/J_G|_{\fc_H})}\ar[rr]^{\ \ \ \cong}\ar[dr]_{\mu_G}&&T^*X\ar[ld]^{\mu_X}\\
    &\fg^*&}\]

\end{prop}
\begin{proof}
It follows from
Lemma \ref{exact} that
\begin{equation}\label{1}
   \bC[G\times J_H]^{J_G|_{\fc_H}}\cong\bC[G\times\fc_H]^{(J_G|_{\fc_H})^{\on{inv}\circ\theta_{0,J}}} 
\end{equation}
We claim that there is an isomorphism
\begin{equation}\label{2}
\bC[G\times\fc_H]^{(J_G|_{\fc_H})^{\on{inv}\circ\theta_{0,J}}}\cong\bC[G\times\fc_X]^{J_X}
\end{equation}
Indeed, by Lemma \ref{exact}, the two
smooth affine schemes $(J_G|_{\fc_H})^{\on{inv}\circ\theta_{0,J}}$
and $J_X$, viewed as subgroup schemes 
$J_G|_{\fc_H=\fc_X}$, are identical over the generic fiber and hence the desired isomorphism follows from Lemma \ref{generic} below. 
Now the isomorphism 
$K\times\fc_X/J_X\cong\fp^{*,r}$ in Lemma \ref{J} and the 
the codimension two property 
$\on{codim}(\fp^*\setminus\fp^{*,r})\geq 2$
in Lemma \ref{regular} (5) 
implies 
\[\bC[G\times J_H]^{J_G|_{\fc_H}}\stackrel{~\eqref{1}}\cong\bC[G\times\fc_H]^{(J_G|_{\fc_H})^{\on{inv}\circ\theta_{0,J}}}\stackrel{~\eqref{2}}\cong
\bC[G\times\fc_X/J_X]\cong
\bC[G\times^K(K\times\fc_X/J_X)]\cong\]
\[\cong\bC[G\times^K\fp^{*,r}]\cong
\bC[G\times^K\fp^{*}]\cong\bC[T^*X].\]
It follows from the construction that corresponding
isomorphism $\overline{(G\times J_H)/J_G|_{\fc_H})}\cong T^*X$ 
is compatible with the $G$-action and the maps 
$\mu_X$ and $\mu_H$.

\end{proof}

\begin{lemma}\label{generic}
Let $\calG\to \mathbb A^n$ be a flat affine group scheme over the affine $n$-space.
Let $\calH_1,\calH_2$ be two flat  affine group subschemes of $\calG$.
Assume $\calH_1=\calH_2$ over the generic point $\eta$ of $\mathbb A^n$.
Then we have 
\[\bC[\calG]^{\calH_1}=\bC[\calG]^{\calH_2}.\]
\end{lemma}
\begin{proof}
   It suffices to show that and $\calH_1$-invariant function $f\in\bC[\calG]^{\calH_1}$
   is also $\calH_2$-invariant.
   Let $a^*_i:\bC[\calG]\to\bC[\calG\times_{\mathbb A^n}\calH_i]$ be the co-action map.
   We have $a_1^*f-f\otimes 1=0$.
   We need to show
   $a_2^*f-f\otimes 1=0$.
   Since the localization map 
   $\bC[\calG\times_{\mathbb A^n}\calH_2]\to\bC[\calG_\eta\times_\eta\calH_{2,\eta}]$ is injective, we reduce to check that $(a_2^*f-f\otimes 1)_\eta=0$
   on $\calG_\eta\times_\eta\calH_{2,\eta}$.
   Now the desired claim follows from the 
assumption that $a^*_1=a^*_2$ over $\eta$
and hence $(a_2^*f-f\otimes 1)_\eta=(a_1^*f-f\otimes 1)_\eta=0$.
\end{proof}

Proposition \ref{Wh=T^*X}  implies that there is a 
$J_H$-action on $T^*X\cong\overline{(G\times J_H)/J_G|_{\fc_H}}$ induced by the multiplication of $J_H$.
Consider the quotient
\[(T^*X\times H)/J_H\cong (T^*X\times_{\fc_H}(\fc_H\times H))/J_H\]
where $T^*X\to T^*X/G\cong\fc_X\cong\fc_H$ is the quotient map.
We have natural map $\mu_{X,H}^\circ:\overline{(T^*X\times H)/J_H}\to\fg^*\times\fc^*$
given by
$\mu_{X,H}(x,c,h)\to (\mu_X(x),\on{Ad}_h(\kappa_H(c))$.
Denote by 
\[\overline{(T^*X\times H)/J_H}=\on{Spec}(\bC[(T^*X\times H)^{J_H})\]
the affine closure. Then $\mu^\circ_{X,H}$ extends to a map 
$\mu_{X,H}:\overline{(T^*X\times H)/J_H}\to\fg^*\times\fh^*$.
Consider the quotient 
\[(G\times H\times\fc_H)/J_G|_{\fc_H}=((G\times\fc_H)\times_{\fc_H} (\fc_H\times H))/J_G|_{\fc_H}\]
where $J_G|_{\fc_H}$ acts on $\fc_H\times H$ via the composed map
$J_G|_{\fc_H}\to J_H\to\fc_H\times H$
 and on $G\times\fc_H$
as before.
We have a natural map 
\[\mu^\circ_{G,H}:(G\times H\times \fc_H)/J_G|_{\fc_H}\to\fg^*\times\fh^*\]
given by $(g,h,c)\to (\on{Ad}_g\kappa_H(c),\on{Ad}_h\kappa_H(c))$.
Denote by 
\[\overline{(G\times H\times\fc_H)/J_G|_{\fc_H})}=\on{Spec}(\bC[G\times H\times\fc_H]^{J_G|_{\fc_H}})\]
the affine closure. The map $\mu^\circ_{G,H}$ extends to a map 
\[\mu_{G,H}:\overline{(G\times H\times\fc)/J_G|_{\fc_H})}\to\fg^*\times\fh^*\]

\begin{corollary}
    \label{kernel}
   There is a horizontal isomorphism of $G\times H$-varieties
\[\xymatrix{\overline{(G\times H\times\fc_H)/J_G|_{\fc_H})}\ar[rr]^{\ \ \ \cong}\ar[dr]_{\mu_{G,H}}&&\overline{(T^*X\times H)/J_H}\ar[ld]^{\mu_{X,H}}\\
    &\fg^*\times\fh^*&}\]   
\end{corollary}
\begin{proof}
  Write $Q=(G\times H\times\fc_H)/J_G|_{\fc_H})$
  and $Q\times_{\fh^*}\fc_H$ be base change of 
  $\mu_{G,H}:Q\to\fg^*\times\fh^*\stackrel{pr}\to\fh^*$
  along the Kostant section $\kappa_H:\fc_H\to\fh^*$.
  Since the image of $\mu_{G,H}$ lands in $\fg^{*,r}\times\fh^{*,r}$, the Kostant-Whittaker descent isomorphism \cite[Proposition 2.3.]{Ga} implies
    \[\bC[Q]\cong\bC[(Q\times_{\fh^*}\fc_H)\times_{\fc_H}(\fc_H\times H)]^{J_H}\]
    On the other hand, 
    $Q\times_{\fh^*}\fc_H\cong (G\times J_H)/J_G|_{\fc_H}$
    and 
    Proposition \ref{Wh=T^*X} implies 
    that 
    \[\bC[(Q\times_{\fh^*}\fc_H)\times_{\fc_H}(\fc_H\times H)]\cong\bC[T^*X\times_{\fc_H}(\fc_H\times H)].\]
Taking $J_H$-invariants, we obtain the desired isomorphism 
\[\bC[Q]\cong \bC[(Q\times_{\fh^*}\fc_H)\times_{\fc_H}(\fc_H\times H)]^{J_H}\cong \bC[T^*X\times_{\fc_H}(\fc_H\times H)]^{J_H}\cong\bC[T^*X\times H/J_H].\]
Compatibility between $\mu_{G,H}$ and $\mu_{X,H}$ follows from the construction.

\end{proof}

\section{Twisted derived Satake equivalence}\label{twisted Satake}
\subsection{Twisted loop groups and their affine Grassmannians}
Let $\check G$ be the simply connected complex dual group of $G$.
Let $F=\bC((t))$ (resp. $F'=\bC((t^2))$) and $\calO=\bC[[t]]$ (resp. $\calO'=\bC[[t^2]]$)
be the Laurent series and power series ring  over $\bC$ with formal variable $t$ (resp. $t^2$).
Denote by
$\check G(F)$ and $\check G(\calO)$
the loop group and arc groups, and 
 $\Gr=\check G(F)/\check G(\calO)$
the affine Grassmannian.
We fix a Borel pair $(\check B,\check T)$
and let $(X^\bullet(\check T),\check\Delta,X_\bullet(\check T),\Delta)$
be the associated root datum, $X^\bullet(\check T)^+$
and $X_\bullet(\check T)^+$ the sets of dominant weights and dominant coweights,
$\rW_{\check G}=N_{\check G}(\check T)/\check T$ the Weyl group, and 
$\check I\subset\check G(\calO)$ the Iwahoric subgroup. It it known that the 
$\check G(\calO)$-orbits (resp. $\check I$-orbits)
on $\Gr$ are parametrized by 
$X_\bullet(\check T)^+$ (resp. $X_\bullet(\check T)$)

Fix a pinning $(\check G,\check B,\check T,\check x)$ of $\check G$ and a pinned involution $\check\theta_0\in\on{Aut}(\check G,\check B,\check T,\check x)$.
Denote by $\check H=\check G^{\check\theta_0}$ 
the symmetric subgroup (Note that $\check H$ is not the dual group of the symmetric subgroup $G^{\theta_0}$ of $G$).
Since
$\check G$ is simply connected, Steinberg's theorem implies 
$\check H$ is connected.
The intersections $(\check B_{\check H}=\check B\cap\check H,\check T_{\check H}=\check T\cap\check H)$ is a Borel pair of $\check H$ and we denote by 
$\rW_{\check H}=N_{\check H}(\check T_{\check H})/\check T_{\check H}$ the Weyl group.

Consider the involution  
$\check\theta: \check G(F)\to \check G(F)$ sending 
$\gamma=\gamma(t)\in \check G(F)$ to 
\begin{equation}\label{ct}
  \check\theta(\gamma)(t)=\check\theta_0(\gamma(-t)).  
\end{equation}
Both $\check G(\calO)$ and  $\check I$ are stable under $\check\theta$ and 
we  denote by 
$\check G(F)^{\check\theta}$,
$\check G(\calO)^{\check\theta}$, 
and $\check I^{\theta}\subset \check G(\calO)^{\check\theta}$ the 
twisted loop group, twisted arc group, and twisted
Iwahori subgroup respectively.
We denote by 
\[\Gr^{\check\theta}=\check G(F)^{\check\theta}/\check G(\calO)^{\check\theta}\]
the quotient known as the twisted affine Grassmannian.
\begin{example}\label{list 2}
(1) If $\check\theta_0$ is trivial then 
$\ct(\gamma)(t)=\gamma(-t)$ and hence
$\check G=\check H$ and $\check G(F)^{\ct}=\check G(F')=\check G(\bC((t^2)))$.

(2)
Here is a list of the the triple
$(\check G,\check H,\check G(F)^{\ct})$ for non-trivial $\check\theta_0$ :
\[(SL_{2n+1},SO_{2n+1}, SU_{2n+1}(F')),\  (SL_{2n},Sp_{2n},SU_{2n}(F')),\  (Spin_{2n},Spin_{2n-1}, Spin_{2n}^{(2)}(F'))\] 
\[(E_6^{sc},F_4, E^{sc,(2)}_6(F')).\]
Here $Spin_{2n}^{(2)}(F')$ and $E_6^{sc,(2)}(F')$ are the quasi-split but non-split forms of 
$Spin_{2n}(F')$ and $E^{sc}_6(F')$.
\end{example}

We recall some basic facts about 
twisted affine Grassmannian and twisted Satake equivalence in
\cite{PR,Z}.
Since $\check G$ is simply connected 
$\Gr^{\ct}$ is connected (\cite[Theorem 1.1]{PR}).
The involution $\check\theta_0$ acts naturally 
on $X_\bullet(\check T)$ and we denote by
$X_\bullet(\check T)_{\check\theta}$ the 
co-invariant and $X_\bullet(\check T)^+_{\check\theta}$ the image of 
$X_\bullet(\check T)^+$ along the natural quotient map 
$X_\bullet(\check T)\to X_\bullet(\check T)_{\check\theta}$. Note that the simply-connectedness of $\check G$ implies that 
$X_\bullet(\check T)_{\check\theta}$ is torsion free.
The Kottwitz homomorphism 
$\check T(F)\to X_\bullet(\check T)_{\check\theta}$
induces an isomorphism
\begin{equation}\label{Kottwitz}
    X_\bullet(\check T)_{\check\theta}\cong \check T(F)^{\check\theta}/T(\calO)^{\check\theta}\ \ \ \ \ [\mu]\to s_{[\mu]}=t^{\mu+\check\theta_0(\mu)}\cdot\check T(\calO)^{\check\theta}
\end{equation}
where  $\mu\in X_\bullet(\check T)$ 
is  a representative
of $[\mu]\in X_\bullet(\check T)_{\check\theta}$.
It is known that the 
assignment 
sending $[\mu]$
to  $\check G(\calO)^{\check\theta}\cdot s_{[\mu]}$ (resp.  $\check I^{\check\theta}\cdot s_{[\mu]}$)
induces a bijection between 
$X_\bullet(\check T)_{\check\theta}^+$ (resp. $X_\bullet(\check T)_{\check\theta}$)
and 
$\check G(\calO)^{\check\theta}$-orbits (resp. $\check I^{\check\theta}$-orbits). 
Here we regard $s_{[\mu]}$
as an element in $\Gr^{\check\theta}$ via the  natural inclusion 
$\check T(F)^{\check\theta}/T(\calO)^{\check\theta}\subset\Gr^{\check\theta}$.

Let $\on{Perv}_{\check G(\calO)^{\ct}}(\Gr^{\ct})$ be the abelian category of $\check G(\calO)^{\check\theta}$-equivariant perverse sheaves on $\Gr^{\check\theta}$. 
In \cite[Theorem 0.1]{Z}, the author proved that 
$\on{Perv}_{\check G(\calO)^{\ct}}(\Gr^{\ct})$ is naturally a Tannakian tensor category with tensor product given by convolution product 
$\star$
 and fiber functor 
 given by hypercohomology 
 $\on{H}^*=\on{H}^*(\Gr^{\check\theta},-):\on{Perv}_{\check G(\calO)^{\check\theta}}(\Gr^{\check\theta})\to\on{Vect}$. In addition, there is a horizontal 
 braided tensor equivalence
 \begin{equation}\label{Sat}
     \xymatrix{(\on{Perv}_{\check G(\calO)^{\ct}}(\Gr^{\ct}),\star)\ar[rr]^{\Phi}_\simeq\ar[dr]_{\on{H}^*}&&(\on{Rep}(H),\otimes)\ar[ld]^{\on{For}}\\
 &\on{Vect}&}
 \end{equation}
where $(\on{Rep}(H),\otimes)$ is the Tannakian tensor category of finite dimensional 
complex representations of $H$.
There is a canonical bijection
\begin{equation}\label{A}
    X^\bullet(T_H)\cong X_\bullet(\check T)_{\check\theta}\ \ \ \ \ \lambda\leftrightarrow [\mu]
\end{equation}
sending 
the set of dominant weights $X^\bullet(T_H)^+$
(with respect to the Borel pair $(B_H,T_H)$)
to $X_\bullet(\check T)_{\check\theta}^+$, such that 
\[\Phi(\IC_{[\mu]})\cong V_\lambda\]
where $\IC_{[\mu]}$ is the $\IC$-complex on
the closure of the spherical orbit $\check G(\calO)^{\check\theta}\cdot s_{[\mu]}$ 
and $V_\lambda$ irreducible representation 
of $H$ with highest weight $\lambda$.

The regular representation $\bC[H]$ of 
$H$ gives rise to an ind perverse sheaf 
on $\IC_{\bC[H]}\in\on{Ind}(\on{Perv}_{\check G(\calO)^{\ct}}(\Gr^{\ct}))$ 
to be called the regular perverse sheaf.

\subsection{Equivariant homology of twisted affine Grassmannians}
The geometric Satake equivalence for 
$\check G$ gives rise to pinning 
$(G,B,T,x)$ of $G$, see \cite[Section 4]{Z}.
Consider the pinned involution $\theta_0\in\on{Aut}(G,T,B,x)$ given by the image of $\check\theta_0\in\on{Aut}(\check G,\check T,\check B,\check x)$
under the isomorphism
\[\on{Aut}(\check G,\check T,\check B,\check x)\cong\on{Out}(\check G)\cong\on{Out}(G)\cong \on{Aut}(G,T,B,x)\]
Note that we are in the setup in Section \ref{symemtric pair} and we will follow the notation in \emph{loc. cit.} and denote by 
$H=G^{\theta_0}$ be the fixed points subgroup and $(B_H=B^{\theta_0},T_H=T^{\theta_0})$ the induced Borel pair.

We denote by $\check\fg, \check\fb, \check\fh$, etc, the Lie algebras and 
$\check\fc=\check\ft//\rW_{\check G}$
and 
$\check\fc_{\check H}=\check\ft_{\check H}//\rW_{\check H}$. 
It follows from Lemma \ref{equ} that 
the canonical isomorphisms
$\check\ft\cong\ft^*$ and $\rW_{\check G}\cong\rW_G$ sends $\check\theta_0$
to $\theta_0$ and the subgroup 
$\rW_{\check H}\subset\rW_{\check G}^{\check\theta_0}$
to the subgroup $\rW_H\subset\rW_G^{\theta_0}$. 
Moreover, the canonical isomorphism 
$\check\fc\cong\check\ft//\rW_{\check G}\cong\ft//\rW_G\cong\fc$
restricts to isomorphisms
$\check\fc_{\check H}\cong\fc_H$.
We remark that the equality 
$\rW_H=\rW^{\theta_0}_G$ in Lemma \ref{equ} (2)
implies 
$\rW_{\check H}=\rW^{\check\theta_0}_{\check G}$.

Consider the equivariant homology 
$\on{H}_*^{\check T}(\Gr)=\on{H}^*_{\check T}(\Gr,\omega)$ (resp.
$\on{H}_*^{\check T_{\check H}}(\Gr^{\check\theta})=\on{H}^*_{\check T_{\check H}}(\Gr^{\check\theta},\omega)$)
and $\on{H}_*^{\check G}(\Gr)=\on{H}^*_{\check G}(\Gr,\omega)$ (resp.
$\on{H}_*^{\check H}(\Gr^{\check\theta})=\on{H}^*_{\check H}(\Gr^{\check\theta},\omega)$)
where $\omega$ is the dualizing sheaf (see, e.g., \cite[Section 2.6]{YZ}).
 Since $\Gr$
 (resp. $\Gr^{\check\theta}$) admits a 
 $\check T$-stable  (resp.
 $\check T_{\check H}$-stable ) stratification with 
 affine space strata given by 
$\check I$-orbits
 (resp.
 $\check I^{\check\theta}$-orbits), the equivariant homology
 $\on{H}_*^{\check T}(\Gr)$ 
 (resp. $\on{H}_*^{\check T_{\check H}}(\Gr^{\check\theta})$ )
 is a free $\bC[\check\ft]\cong\on{H}^*_{\check T}(pt)$-module (resp. a free $\bC[\check\ft_{\check G}]\cong\on{H}^*_{\check T_{\check H}}(pt)$-module). 
It follows that 
$\on{H}_*^{\check G}(\Gr)\cong
\on{H}_*^{\check T}(\Gr)^{\rW_{\check G}}$
(resp. $
\on{H}_*^{\check H}(\Gr^{\check\theta})\cong
\on{H}_*^{\check T_{\check H}}(\Gr^{\check\theta})^{\rW_{\check H}}$ )
 is a free over $\bC[\check\fc]\cong\bC[\check\ft]^{\rW_{\check G}}\cong\on{H}_{\check G}^*(pt)$ (resp.  over $\bC[\check\fc_{\check H}]\cong\bC[\check\ft_{\check H}]^{\rW_{\check H}}\cong\on{H}_{\check H}^*(pt)$).

We shall recall the construction of a Hopf algebra structure on
$\on{H}_*^{\check T}(\Gr)$ (resp. $\on{H}_*^{\check T_{\check H}}(\Gr^{\check\theta})$).
Pick a real conjugation $\check\eta_0$ of $\check G$ such that 
$\check\eta_c:=\check\eta_0\check\theta_0=\check\theta_0\check\eta_0$ is a compact involution, that is, $\check G_c=\check G^{\check\eta_c}$
is compact.
Let $\check T_c=\check T\cap\check G_c$,
$\check H_c=\check G_c\cap\check H$, $\check T_{\check H_c}=\check T\cap \check H_c$.
Let $\Omega \check G_c\subset\check G(F)$
be the subgroup of base polynomial map 
$(S^1,1)\to (\check G_c,e)$. 
Then we have the Gram-Schmidt factorization
$\check G(F)\cong\Omega G_c\times \check G(\calO)$ and a $\check G_c$-equivariant homeomorphism 
$\Gr\cong\Omega G_c$. 
Taking $\check\theta$-fixed points, we obtain a  
factorization of the twisted loop group
$\check G(F)^{\check\theta}\cong(\Omega\check G_c)^{\check\theta}\times\check G(\calO)^{\check\theta}$
and a $\check K_c$-eqivariant homeomorphism $\Gr^{\check\theta}\cong(\Omega \check G_c)^{\check\theta}$.
There is a graded Hopf algebra structure on 
$\on{H}_*^{\check T_c}((\Omega \check G_c)^{\check\theta})$ (resp. $\on{H}_*^{\check T_{\check H_c}}((\Omega \check G_c)^{\check\theta})$)
over $\bC[\check\ft]$ (resp. over $\bC[\check\ft_{\check H}]$)
with co-product and product 
given by 
 push-forward along the diagonal embedding 
 $\Delta:\Omega G_c\to \Omega G_c\times \Omega G_c$ and 
 and the group multiplication 
 on $m:\Omega G_c\times \Omega G_c\to \Omega G_c$ of $\Omega G_c$ (resp. $(\Omega G_c)^{\check\theta}$)
and we can transport the Hopf algebra structure to $\on{H}_*^{\check T}(\Gr)$ (resp. $\on{H}_*^{\check T_{\check H}}(\Gr^{\check\theta})$) via 
the above equivariant homeomorphisms.
The same construction endow $\on{H}_*^{\check G}(\Gr)$ (resp. $\on{H}_*^{\check H}(\Gr^{\check\theta})$) a graded Hopf algebra structure over $\bC[\check\fc]$ (resp. $\bC[\check\fc_{\check H}]$) 

Consider the embedding 
$\check\rho:\Gr^{\check\theta}\cong(\Omega G_c)^{\check\theta}\to\Gr\cong\Omega G_c$. 
Since $\check\rho$ is a group homomorphism,
it gives rise to a Hopf algebra
map
\begin{equation}\label{rho_*}
   \rho_*^{\check T_{\check H}}:\on{H}_*^{\check T_{\check H}}(\Gr^{\check\theta})\to\on{H}_*^{\check T_{\check H}}(\Gr)\cong\on{H}_*^{\check T}(\Gr)\otimes_{\bC[\check\ft]}\bC[\check\ft_{\check H}] 
\end{equation}
\begin{lemma}\label{commutativity}
(1) The Hopf algebra map $\rho_*^{\check T_{\check H}}$~\eqref{rho_*} is injective.

(2) The Hopf algebra  $\on{H}_*^{\check T}(\Gr)$ 
  (resp. $\on{H}_*^{\check T_{\check H}}(\Gr^{\check\theta})$) is commutative
  and cocommutative.

\end{lemma}
\begin{proof}
  Proof of (1). 
 Since  
$\check T_{\check H}$ contains regular semisimple elements of $\check G$, see, e.g., \cite[Proposition 6.60]{Kn},  the $\check T_{\check H}$-fixed points on $\Gr$ 
  and $\Gr^{\check\theta}$ are given by 
   $(\Gr)^{\check T_{\check H}}=\check T(F)/\check T(\calO)\cong X_\bullet(\check T)$ 
   and $(\Gr^{\check\theta})^{\check T_{\check H}}=
   \check T(F)^{\check\theta}/\check T(\calO)^{\check\theta}\cong X_\bullet(\check T)_{\check\theta}$ 
and localization theorem
implies that there are vertical isomorphism
in following diagram commutes
   \begin{equation}\label{localization rho}
       \xymatrix{\bC(\check\ft_{\check H})\otimes_{\bC[\check\ft_{\check H}]}\on{H}_*^{\check T_{\check H}}(\Gr^{\check\theta})\ar[r]^{}\ar[d]^\simeq&\bC(\check\ft_{\check H})\otimes_{\bC[\check\ft_{\check H}]}\on{H}_*^{\check T_{\check H}}(\Gr)\ar[d]^\simeq\\
   \bC(\check\ft_{\check H})[X_\bullet(\check T)_{\check\theta}]\ar[r]&\bC(\check\ft_{\check H})[X_\bullet(\check T)]}
   \end{equation}
  here $\bC(\check\ft_{\check H})$
  is the fraction field of $\bC[\check\ft_{\check H}]$, the upper horizontal arrow is 
  $\rho_*^{\check T_{\check H}}\otimes_{\bC[\check\ft_{\check H}]}\bC(\check\ft_{\check H})$,
and the lower horizontal arrow is induced by the 
natural embedding of fixed points
\[(\Gr^{\check\theta})^{\check T_{\check H}}\cong X_\bullet(\check T)_{\check\theta}\to
\Gr^{\check T_{\check H}}\cong X_\bullet(\check T)\]
sending $[\check\mu]$ to $\check\mu+\check\theta_0(\check\mu)$ (see ~\eqref{Kottwitz}).
Since both $\on{H}_*^{\check T}(\Gr)$
and $\on{H}_*^{\check T_{\check H}}(\Gr^{\check\theta})$ are flat $\bC[\check\ft_{\check H}]$-modules it implies 
$\rho_*$ is injective.

The commutative
  and cocommutative of 
$\on{H}_*^{\check T}(\Gr)$
follows from the fact that 
$\Gr\cong\Omega G_c$ is a homotopic commutative $\on{H}$-space. Now the commutative
  and cocommutative of 
  $\on{H}_*^{\check T_{\check H}}(\Gr^{\check\theta})$ follows from (1).

\end{proof}

It follows from Lemma
\ref{commutativity} that the Hopf algebra
$\on{H}_*^{\check G}(\Gr)$
(resp. $\on{H}_*^{\check H}(\Gr^{\check\theta})$)
is commutative and cocommutative and we denote by 
$\on{Spec}(\on{H}_*^{\check G}(\Gr))$ 
(resp. $\on{Spec}(\on{H}_*^{\check H}(\Gr^{\check\theta}))$) the associated 
affine flat commutative group scheme over 
$\check \fc_{}\cong\fc$ (resp. $\check \fc_{\check H}\cong\fc_H$).
It is shown in \cite{YZ} that there is an isomorphism of group schemes
\[
\on{Spec}(\on{H}_*^{\check G}(\Gr))\cong J_G.\]
Consider the map 
\begin{equation}\label{rho}
\rho^{\check H}=\on{Spec}(\rho^{\check H}_*):J_G|_{\fc_H}\cong
\on{Spec}(\on{H}_*^{\check H}(\Gr))\to  \on{Spec}(\on{H}_*^{\check H}(\Gr^{\check\theta})) 
    \end{equation}
between group schemes over $\fc_H$.

\begin{prop}\label{cohomology}
(1) There is an isomorphism
$\on{Spec}(\on{H}_*^{\check H}(\Gr^{\check\theta}))\cong J_H$
of group schemes over $\check\fc_{\check H}\cong\fc_H$.

(2)
 The following diagram commutes
    \[\xymatrix{\on{Spec}(\on{H}_*^{\check H}(\Gr))\ar[r]^{\rho^{\check H}}\ar[d]^\simeq&\on{Spec}(\on{H}_*^{\check H}(\Gr^{\check\theta}))\ar[d]^\simeq\\
J_G|_{\fc_H}\ar[r]^{\on{Nm}}&J_{H}}\]
    where $\on{Nm}$ is the norm map 
    $\on{Nm}(g)=g\theta_{0,J}(g)$ in~\eqref{norm}.
    
\end{prop}

\begin{proof}
We shall construct a surjective map
\begin{equation}\label{psi_*}
\psi^{\check H}_*:\bC[J_G|_{\fc_H}]\cong\on{H}_*^{\check H}(\Gr)\ra\on{H}_*^{\check H}(\Gr^{\check\theta})
\end{equation} of Hopf algebras such that  the composition 
\begin{equation}\label{Nm property}
\rho_*^{\check H}\circ\psi_*^{\check H}=\on{Nm}^{*}:\bC[J_G|_{\fc_H}]\to
\on{H}_*^{\check H}(\Gr^{\check\theta})\to \bC[J_G|_{\fc_H}]
\end{equation}
is given by the pull-back along the  map
$\on{Nm}:J_G|_{\fc_H}\to J_G|_{\fc_H}$.
Since $\rho_*$ is injective,
it follows that $\on{Spec}(\on{H}_*^{\check H}(\Gr^{\check\theta}))$ is isomorphic to the image $\on{Spec}(\on{H}_*^{\check H}(\Gr^{\check\theta}))\cong\on{Im}(\Nm)$.
On the other hand,
Lemma \ref{exact} implies that 
$\on{Im}(\Nm)=J_H$ and altogether we obtain 
\[\on{Spec}(\on{H}_*^{\check H}(\Gr^{\check\theta}))\cong J_H\]
and $\rho\cong\on{Nm}$.
Part (1) and (2) follow.

To  construct~\eqref{psi_*}, consider the 
Beilinson-Drinfeld Grassmannian 
$f:\widetilde\Gr\to\mathbb A^1$ 
in \cite[Section 2]{Z} 
such that 
$\widetilde\Gr|_{\mathbb G_m}\cong\Gr\times\mathbb G_m$
and $\widetilde\Gr|_{0}\cong\Gr^{\check\theta}$.
The nearby cycle functor 
associated to the family $f$ gives rise to a functor
\begin{equation}\label{phi}
     \phi_f:D_H(\Gr)\to D_H(\Gr^{\check\theta}).
\end{equation}
 between the $\check H$-equivariant derived categories of constructible sheaves.
We claim that there is a natural map
$\phi_f(\omega_\Gr)\to\omega_{\Gr^{\check\theta}}$ which induces the desired map
\begin{equation}\label{phi}
    \psi_*^{\check H}:\on{H}_*^{\check H}(\Gr)
\cong
\on{H}^*_{\check H}(\Gr^{\check\theta},\phi_f\omega)
\to\on{H}^*_{\check H}(\Gr^{\check\theta},\omega)\cong\on{H}_*^{\check H}(\Gr^{\check\theta})
\end{equation}
Indeed, by \cite[Lemma 4.8]{CMNO}, there is an isomorphism $\phi^{\check H}_f\cong \psi^{\check H_c}_*\cong\psi^{\check H_c}_!$ where 
$\psi:\Gr\to\Gr^{\check\theta}$ is 
a $H_c$-equivariant specialization map.
The desired map is given by the adjunction 
\[\phi_f^{\check H}(\omega_\Gr)\cong
\psi^{\check H_c}_*\omega_{\Gr}\cong\psi^{\check H_c}_!\omega_{\Gr}\to\omega_{\Gr^{\check\theta}}.\]

We show that $\psi_*^{\check H}$~\eqref{phi}
satisfies the equality~\eqref{Nm property}.
Since $\on{H}_*^{\check H}(\Gr)\cong\bC[J_G|_{\fc_H}]$ is free over $\bC[\fc_H]$, it suffices to check
the equality
\begin{equation}\label{loc c}
(\rho^{\check H}_*\circ\psi^{\check H}_*)\otimes\bC(\fc_{H})=(\on{Nm}^*)\otimes\bC(\fc_H)
\end{equation}
after tensoring the fraction field $\bC(\fc_H)$.
It follows from \cite[Lemma 5.7]{Z}
that there is a commutative diagram
\begin{equation}\label{localization phi}
       \xymatrix{\bC(\check\ft_{\check H})\otimes_{\bC[\check\ft_{\check H}]}\on{H}_*^{\check T_{\check H}}(\Gr)\ar[r]^{}\ar[d]^\simeq&\bC(\check\ft_{\check H})\otimes_{\bC[\check\ft_{\check H}]}\on{H}_*^{\check T_{\check H}}(\Gr^{\check\theta})\ar[d]^\simeq\\
   \bC(\check\ft_{\check H})[X_\bullet(\check T)]\ar[r]&\bC(\check\ft_{\check H})[X_\bullet(\check T)_{\check\theta}]}
   \end{equation}
where the upper arrow is 
$\psi_*^{\check T_{\check H}}\otimes\bC(\check\ft_H)$ and
the bottom arrow is induced by the map 
$X_\bullet(\check T)\to X_\bullet(\check T)_{\check\theta_0}$ sending $\check\mu$ to it image 
$[\check\mu]$. Combining with the description of 
$\rho_*^{\check T_{\check H}}\otimes\bC(\check\ft_H)$ in 
~\eqref{localization rho}, we see that 
\begin{equation}\label{loc t}
(\rho^{\check T_{\check H}}_*\circ\psi^{\check T_{\check H}}_*)\otimes\bC(\check\ft_{\check H})=(\on{Nm}^*)\otimes\bC(\check\ft_{\check H}):
\bC(\check\ft_{\check H})\otimes_{\bC[\check\ft_{\check H}]}\on{H}_*^{\check T_{\check H}}(\Gr)\to\bC(\check\ft_{\check H})\otimes_{\bC[\check\ft_{\check H}]}\on{H}_*^{\check T_{\check H}}(\Gr)
\end{equation}
Taking $\rW_H$-invariants of~\eqref{loc t} 
we get the desired equality in~\eqref{loc c}

We show that $\psi^{\check H}_*$ in ~\eqref{phi} is surjective. It suffices to check 
$\psi_*:\on{H}_*(\Gr)\to\on{H}_*(\Gr^{\check\theta})$ is surjective for non-equivariant homology because 
any homogeneous lifting of 
a $\bC[\fc_H]$-basis of $\on{H}_*^{\check H}(\Gr^{\check\theta})$ we can choose these lifting to be in the image  of 
$\psi^{\check H}_*$ since $\on{H}^{\check H}_*(\Gr)\to \on{H}_*(\Gr)$ is surjective.
Instead, we will show that the dual map
$\psi^*:\on{H}^*(\Gr^{\check\theta})\to\on{H}^*(\Gr)$
is injective.
Let $\psi^*(a)\in \on{H}^*(\Gr^{})$ be the image of a non-zero 
$a\neq0\in\on{H}^*(\Gr^{\check\theta})$.
It follows from \cite{G1,YZ} that 
, to check that  $\psi^*(a)\neq0$, it suffices to show that 
there is a $\IC$-complex $\IC_\mu\in\on{Perv}_{\check G(\calO)}(\Gr)$ such that $\psi^*(a)$ acts non-trivially on 
$\on{H}^*(\Gr,\IC_\mu)$. This is equivalent to check that the action of $a$ on 
$\on{H}^*(\Gr,\IC_\mu)\cong \on{H}^*(\Gr^{\check\theta},\psi^{}_*\IC_\mu)$
is non-trivial. According to  \cite{Z}, 
$\psi^{}_*\IC_\mu$ is a direct sum of 
$\IC_{[\lambda]}\in\on{Perv}_{\check G(\calO)^{\check\theta}}(\Gr^{\check\theta})$
and we reduce to check that 
$a$ acts non-trivially on $\on{H}^*(\Gr^{\check\theta},\IC_{[\lambda]})$
for some $\IC_{[\lambda]}$. 
In the case when $\theta_0$ is trivial, 
so that $\Gr^{\check\theta}\cong\Gr$,
this follows from a general fact in  \cite[Proposition 4.3.3]{G1},
using general properties of 
intersections of $\check G(\calO)$-orbits and 
$\check I_\infty$-orbits  on 
$\Gr$, where $\check I_\infty=\on{ker}(\check G[t^{-1}]\to\check B)$ is the polynomial Iwahori subgorup.\footnote{The polynomial Iwahori subgorup is denoted by $I$ in \emph{loc. cit.}}
The same argument works in the twisted case
$\check\theta_0\neq\on{id}$ where one uses the 
twisted polynomial Iwahori subgorup  
$\check I_\infty^{\ct}=\check I_\infty\cap G(F)^{\check\theta}$.

\end{proof}

\begin{remark}\label{[2]}
Here is an alternative argument for Proposition \ref{cohomology} in 
 the case $\check\theta_0$ is trivial.
 Under the homeomorphisms
$\Gr\cong\Omega G_c$ and 
  $\Gr^{\check\theta}=\check G(F')/\check G(\calO')\cong\Omega \check G_c$ the 
  inclusion $\check\rho:\Gr^{\check\theta}\to\Gr$
  becomes the square map 
  $[2]:\Omega \check G_c\to\Omega \check G_c$ sending $\gamma \to \gamma^2$. 
  Since the square map factors as  
  $[2]:\Omega\check G_c\stackrel{\Delta}\to\Omega \check G_c\times \Omega \check G_c\stackrel{m}\to\Omega \check G_c$, the claim follows from the fact that the 
  coproduct and product on $\bC[J_G]\cong\on{H}_*^{\check G_c}(\Omega\check G_c)$
  are given by push-forward along $\Delta$
  and $m$.
\end{remark}

\subsection{Monoidal structures on equivariant cohomology}
Let $\on{Lie}J_H$ be the Lie algebra of 
$J_H$ and $U(\Lie J_H)$ its universal enveloping algebra.
Then Proposition \ref{cohomology}
implies that 
there is an isomorphism 
$\on{H}^*_{\check H}(\Gr^{\ct})\cong U(\Lie J_H)$
such that the paring between cohomology and homology of $\Gr^{\ct}$
becomes the paring between universal enveloping algebra and ring of functions for the group
scheme $J_H$, see \cite[Section 4.6.2]{CMNO}.

Let $D^b_{\check G(\calO)^{\ct}}(\Gr^{\ct})$
be the $\check G(\calO)^{\ct}$-equivariant derived category of 
$\Gr^{\ct}$.
Consider the equivariant cohomology fucntor
\[\on{H}^*_{\check H}(\Gr^{\ct},-):D^b_{\check G(\calO)^{\ct}}(\Gr^{\ct})\to 
D^b(\on{H}^*_{\check H}(pt)\on{-mod})\cong
D^b(\bC[\fc_H]\on{-mod}).\]
It follows from Kunneth formula that $\on{H}^*_{\check H}(\Gr^{\ct},-)$ has a monoidal 
structure 
\begin{equation}\label{monoidal}
    \on{H}^*_{\check H}(\Gr^{\ct},\calF\star\calF')\cong 
\on{H}^*_{\check H}(\Gr^{\ct},\calF)\otimes_{\bC[\fc_H]}
\on{H}^*_{\check H}(\Gr^{\ct},\calF')
\end{equation}
see \cite[Proposition 3.4.1]{G1}.
Since $\on{H}^*_{\check H}(\Gr^{\ct})\cong U(\Lie J_H)$
is  generated by the primitive elements $\Lie J_H$, the isomorphism~\eqref{monoidal} 
is compatible with the natural $\on{H}^*_{\check H}(\Gr^{\ct})$-actions, see \cite[Lemma 4.2.1]{G1}.
Moreover, the same argument as in \cite[Lemma 3.1]{YZ} shows that the $\Lie J_H$-action
on $\on{H}^*_{\check H}(\Gr^{\ct},\calF)$
can be exponentiated to a 
$J_H$-action.

\subsection{Derived Satake equivalence
for the twisted affine Grassmannians}
Let $D_{\check G(\calO)^{\ct}}(\Gr^{\ct})$
be the co-complete monoidal
dg category of 
$\check G(\calO)^{\ct}$-equivariant 
 constructible complexes on 
$\Gr^{\ct}$
with monoidal structure given by the 
convolution product.
It is known that $D_{\check G(\calO)}(\Gr)$ is compactly generated and we denote by $D_{\check G(\calO)}(\Gr)^c$
the subcategory of compact objects.
The equivariant derived category
$D^b_{\check G(\calO)^{\ct}}(\Gr^{\ct})\subset D_{\check G(\calO)^{\ct}}(\Gr^{\ct})$ coincides with the 
subcategories of 
locally compact objects\footnote{An object 
$\calF\in D_{\check G(\calO)^{\ct}}(\Gr^{\ct})$ is called locally compact if $\on{For}(\calF)\in\in D_{}(\Gr^{\ct})$ is compact.} and we denote by $\on{Ind}(D^b_{\check G(\calO)}(\Gr))$ the ind-completion of $D^b_{\check G(\calO)}(\Gr)$.
We denote by $D_{\check G(\calO)}(\Gr), D_{\check G(\calO)}(\Gr)^c$, etc, the corresponding categories for $\Gr$.

Denote by
$D^{G}(\on{Sym}(\fg[-2]))$ be the monoidal dg-category of  $G$-equivariant dg-modules 
over the dg-algebra $\on{Sym}(\fg[-2])$ (equipped with trivial differential) with monoidal structure given by (derived) tensor prodcut
$(\mF_1,\mF_2)\to \mF_1\otimes\mF_2:=
\mF_1\otimes^L_{\on{Sym}(\fg[-2])}\mF_2$.
It is known that $D^{G}(\on{Sym}(\fg[-2]))$ is compactly generated and the full subcategory 
$D^{G}(\on{Sym}(\fg[-2]))^c$ of compact objects 
coincides with 
be the full subcategory $D^{G}(\on{Sym}(\fg[-2]))^c=D^{G}_{\on{perf}}(\on{Sym}(\fg[-2]))$
consisting of perfect modules.
Denote by $D^{G}(\on{Sym}(\fg[-2]))_{\mathcal N_{\fg^*}}$
and  $D^{G}_{\on{perf}}(\on{Sym}(\fg[-2]))_{\calN_{\fg^*}}$
the full subcategory of $D^{G}(\on{Sym}(\fg[-2]))$
and $D^{G}_{\on{perf}}(\on{Sym}(\fg[-2]))$ respectively
consisting of modules that are set theoretically supported on the 
nilpotent cone $\calN_{\fg^*}$.
We denote by $D^H(\Sym(\fh[-2]))$, etc, the corresponding categories for $H$.

The derived Satake equivalence in \cite{BF,AG} asserts that there are monoidal equivalences
\begin{equation}\label{DSat}
 \Phi^b:D^b_{\check G(\calO)}(\Gr)\cong D^{G}_{\on{perf}}(\on{Sym}(\fg[-2]))\ \ 
 \on{Ind}\Phi^b:\on{Ind}(D^b_{\check G(\calO)}(\Gr))\cong D^{G}(\on{Sym}(\fg[-2]))
\end{equation}
\begin{equation}\label{DSat2}
\Phi^c:D_{\check G(\calO)}(\Gr)^c\cong D^{G}_{\on{perf}}(\on{Sym}(\fg[-2]))_{\calN_{\fg^*}},\ \
\Phi:D_{\check G(\calO)}(\Gr)\cong D^{G}(\on{Sym}(\fg[-2]))_{\calN_{\fg^*}}
\end{equation}
Note that $\Phi\cong\on{Ind}(\Phi^c)$.

We have the following version of derived Satake equivalence for $\Gr^{\ct}$:
\begin{thm}\label{twisted DSat}
(1) There is a canonical monoidal equivalence 
\[
 \on{Ind}\Phi^b_{\ct}:\on{Ind}(D^b_{\check G(\calO)^{\ct}}(\Gr^{\ct}))\cong D^{H}(\on{Sym}(\fh[-2]))\]
 which restricts to a monoidal equivalence
\[
 \Phi^b_{\ct}:D^b_{\check G(\calO)^{\ct}}(\Gr^{\ct})\cong D^{G}_{\on{perf}}(\on{Sym}(\fh[-2]))
 \]
 on compact objects.

 (2)There is a canonical monoidal equivalence 
 \[
\Phi_{\ct}:D_{\check G(\calO)^{\ct}}(\Gr^{\ct})\cong D^{H}(\on{Sym}(\fh[-2]))_{\calN_{\fh^*}}
\]
which restricts to a monoidal equivalence
 \[
\Phi^c_{\ct}:D_{\check G(\calO)^{\ct}}(\Gr^{\ct})^c\cong D^{H}_{\on{perf}}(\on{Sym}(\fh[-2]))_{\calN_{\fh^*}}\]
on compact objects.
\end{thm}
\begin{proof}
Write $\calC:=\on{Ind}(D_{\check G(\calO)^{\ct}}(\Gr^{\ct}))$.
The dg-category $\calC$
is a module category for the dg-category 
$D(\on{QCoh}(BH))$ of quasi-coherent sheaves on the classifying stack $BH$ where the  monoidal action is given by the  
the Satake equivalence~\eqref{Sat}.
We can 
form the de-equivariantized category $\calC_{\on{deeq}}:=\calC\times_{BH}\{\on{pt}\}$ 
with objects $\on{Ob}(\calC_{\on{deeq}})=\on{Ob}(\calC)$ and (dg)-morphisms
\[\on{Hom}_{\calC_{\on{deeq}}}(\mF_1,\mF_2)=\on{Hom}_{\calC}(\mF_1,\mF_2\star\bC[H])=
\on{RHom}_{\on{Ind}(D_{\check G(\calO)^{\ct}}(\Gr^{\ct}))}(\mF_1,\mF_2\star\bC[H])).\]
Every object $\mF\in\calC_{\on{deeq}}$ carries a natural action of $H$ and we can recover $\calC$ by taking $H$-equivariant objects in $\calC_{\on{deeq}}$.
The fact that $\IC_0$ is compact and generates $\calC$ under the action of $D\on{QCoh}(BH)$
implies that $\IC_0$, viewed as an object in $\calC_{\on{deeq}}$, is a compact generator.
Hence the Barr-Beck-Lurie theorem implies 
that the assignment $\mF\to \on{Hom}_{\calC_{\on{deeq}}}(\IC_0,\mF)$
defines  an equivalence of categories 
\[\calC_{\on{deeq}}\is D(\on{Hom}_{\calC_{\on{deeq}}}(\IC_0,\IC_0)^{\on{op}})\ \ \  \ (resp.\ \  \calC_{}
\is D^{G_n}(\on{Hom}_{\calC_{\on{deeq}}}(\IC_0,\IC_0)^{\on{op}}))
\]
where $\on{Hom}_{\calC_{\on{deeq}}}(\IC_0,\IC_0)^{op}$ is the opposite of the dg-algebra of endomorphism 
of $\IC_0$ and 
$D(\on{Hom}_{\calC_{\on{deeq}}}(\IC_0,\IC_0)^{\on{op}})$
(resp. $D^{G_n}(\on{Hom}_{\calC_{\on{deeq}}}(\IC_0,\IC_0)^{\on{op}})$)
are the corresponding dg-categories of dg-modules (resp. $H$-equivariant dg-modules).
Using the pointwise purity of $\IC$-complexes in \cite[Lemma 1.1]{Z}, 
one can apply the standard results in
\cite[Section 6.5]{BF} (see Theorem \ref{formality} for a more general formality resutls) to conclude
that the dg-algebra $\on{Hom}_{\calC_{\on{deeq}}}(\IC_0,\IC_0)^{\on{op}}$ is formal and hence there is 
$H$-equivariant
isomorphism
\[\on{Hom}_{\calC_{\on{deeq}}}(\IC_0,\IC_0)^{\on{op}}\is\on{RHom}_{\on{Ind}(D_{\check G(\calO)^{\ct}}(\Gr^{\ct}))}(\IC_0,\IC_0\star\bC[H])^{op}\is\on{Ext}^*_{\on{Ind}(D_{\check G(\calO)^{\ct}}(\Gr^{\ct}))}(\IC_0,\IC_0\star\bC[H])^{op}.\]
Using 
the parity vanishing of $\IC$-complexes 
and surjectivity (resp. injectivity)
of equivariant stalks (resp. equivariant co-stalks) in
\cite[Lemma 1.1 and Lemma 5.8]{Z}, one can apply the
fully-faithfulness property 
of equivariant cohomology functor
\cite[Lemma 4.20]{CMNO} or \cite{G1} 
to conclude that there is a $H$-equivariant isomorphism of graded algebras
\[\on{Ext}^*_{\on{Ind}(D_{\check G(\calO)^{\ct}}(\Gr^{\ct}))}(\IC_0,\IC_0\star\bC[H])^{op}\cong\on{Hom}^\bullet 
_{\on{H}^*_{\check H}(\Gr^{\ct})}(\on{H}^*_{\check H}(\IC_0),\on{H}^*_{\check H}(\IC_0\star\bC[H]))^{op}\cong\]
\[\cong\on{Hom}^\bullet 
_{\on{H}^*_{\check H}(\Gr^{\ct})}(\on{H}^*_{\check H}(\IC_0),\on{H}^*_{\check H}(\IC_0)\otimes_{\bC[\fc_H]} \on{H}^*_{\check H}(\bC[H]))^{op}
\cong\on{Hom}^\bullet 
_{\on{H}^*_{\check H}(\Gr^{\ct})}(\bC[\fc_H],\bC[\fc_H\times H])^{op}\]
Note that 
\[\on{Hom}^\bullet 
_{\on{H}^*_{\check H}(\Gr^{\ct})}(\bC[\fc_H],\bC[\fc_H\times H])^{op}\cong(\bC[\fc_H\times H]^{\on{H}^*_{\check H}(\Gr^{\ct})})^{op}\cong(\bC[\fc_H\times H]^{\on{Spec}(\on{H}_*^{\check H}(\Gr^{\ct}))})^{op}\cong\]
\[\stackrel{(1)}\cong
(\bC[\fc_H\times H]^{J_H})^{op}\stackrel{(2)}\cong
(\bC[\fh^{*,r}[2]])^{op}\stackrel{(3)}\cong(\bC[\fh^{*}[2]])^{op}\stackrel{(4)}\cong\bC[\fh^{*}[2]]\cong\Sym(\fh[-2]).\]
Here (1) follows from Proposition \ref{cohomology}, (2) follows from Lemma \ref{J}(3)
and the fact that under the isomorphism
$\on{H}^*(\Gr^{\ct},\IC_0\star\bC[H])\cong \bC[H]$ the cohomological grading on 
$\bC[H]$ 
is induced by the right $\bG_m$-action 
on $H$ through the co-character $\bG_m\subset SL_2\stackrel{\phi}\to H$, see \cite[Theorem 5.1]{Z}, (3) because 
$\fh^*\setminus\fh^{*,r}$ has codimension $\geq2$,
and (4)
because $\bC[\fh^*[2]]$ 
is commutative with grading in even degree. 
All together,
we conclude that there is an equivalence
\begin{equation}\label{formula}
    \on{Ind}\Phi^b_{\ct}:\on{Ind}(D^b_{\check G(\calO)^{\ct}}(\Gr^{\ct}))\is D^H(\Sym(\fh[-2]))
\end{equation}
\[\mF \to \on{RHom}_{\on{Ind}(D^b_{\check G(\calO)^{\ct}}(\Gr^{\ct}))}(\IC_0,\mF\star\bC[H]).\]
Passing to subcategories of compact objects 
we obtain 
\[\Phi_{\ct}^b:D^b_{\check G(\calO)^{\ct}}(\Gr^{\ct})\cong D^{G}_{\on{perf}}(\on{Sym}(\fh[-2])).\]
Assuming the compatibility with 
nearby cycles functors in Theorem \ref{Comp}, the monoidal properties follows from the general discussion in \cite[Section 5.4]{CMNO}.

Part (2) follows from the general discussion in \cite[Section 12]{AG} (see also Theorem \ref{twisted DSat}
for the proof 
in the similar (but more involved) setting of 
derived Satake for loop symmetric spaces).

\end{proof}

\subsection{Spectral description of nearby cycles functors}
Consider the functor 
\[\phi_{f}:\on{Ind}D^b_{\check G(\calO)}(\Gr)\stackrel{\on{For}}\to \on{Ind}D_{\check G(\calO)^{\ct}}(\Gr)\stackrel{\phi_f}\to \on{Ind}D_{\check G(\calO)^{\ct}}(\Gr^{\ct})\]
where $\on{For}$ is the forgetful functor and 
$\phi_f$ is the nearby cycle functor in~\eqref{phi}.
Consider the functor 
\[p_*:D^{G}(\on{Sym}(\fg[-2]))\to D^{H}(\on{Sym}(\fh[-2]))\]
induced by the projection
$p:\fg=\fh\oplus\fk\to\fh$.
We have the following spectral description of the nearby cycle functor 
$\phi_f$:

\begin{prop}\label{Comp}
The following commutative diagram
 \[\xymatrix{\on{Ind}(D^b_{\check G(\calO)}(\Gr))\ar[r]^{\phi_f}\ar[d]_{}^\simeq&\on{Ind}(D^b_{\check G(\calO)^{\ct}}(\Gr^{\ct}))\ar[d]^\simeq\\
D^G(\Sym(\fg[-2]))\ar[r]^{p_*}&\D^H(\Sym(\fh[-2]))}\]
\end{prop}
\begin{proof}
Consider the functor 
$A:=\on{Ind}(\Phi_{\ct}^b)\circ\phi_f\circ\on{Ind}(\Phi^b_{})^{-1}:D^G(\Sym(\fg[-2]))\to D^H(\Sym(\fh[-2]))$.
Since $D^G(\Sym(\fg[-2]))$ is generated by $\Sym(\fg[-2])\otimes\bC[G]$
and $A$ is a continous functor, to show that $A\cong p_*$, it suffices to construct a functorial isomorphism
$A(\Sym(\fg[-2])\otimes\bC[G])\cong p_*(\Sym(\fg[-2])\otimes\bC[G])$.
Note that we have
\[p_*(\Sym(\fg[-2]\otimes\bC[G])\cong\Sym(\fh[-2])\otimes\on{Res}^G_H\bC[G]\]
On the other hand, 
\[A(\Sym(\fg[-2])\otimes\bC[G])\cong
\on{Ind}(\Phi_{\ct}^b)\circ\phi_f\circ\on{Ind}(\Phi^b_{})^{-1}(\Sym(\fg[-2]\otimes\bC[G])\cong\]
\[\cong 
\on{Ind}\Phi_{\ct}^b\circ\phi_f(\IC_0\star\bC[G])\cong
\on{Ind}\Phi_{\ct}^b(\IC_0\star\on{Res}^G_H(\bC[H]))\cong\Sym(\fh[-2])\otimes
\on{Res}^G_H(\bC[H])).\]
Now the desired isomorphism is given by 
\[A(\Sym(\fg[-2])\otimes\bC[G])\cong\Sym(\fh[-2])\otimes
\on{Res}^G_H(\bC[H]))\cong p_*(\Sym(\fg[-2])\otimes
\bC[H]).\]

\end{proof}

\subsection{Purity and formality}
Let $D^{b,mix}_{\check G(\calO)^{\ct}}(\Gr^{\ct})$
be the mixed equivariant derived category.
We say that an object 
$\calF\in D^b_{\check G(\calO)^{\ct}}(\Gr^{\ct})$
is  $!$-pure (resp. $*$-pure) if 
$\calF$ has a lifting 
$\tilde\calF\in D^{b,mix}_{\check G(\calO)^{\ct}}(\Gr^{\ct})$ such that, 
for all spherical orbits
$j_{\lambda}:\check G(\calO)\cdot t^\lambda\subset\Gr$, 
the $m$-th cohomology sheaf
$H^m(j_\lambda^!\tilde\calF)$ 
(resp. $H^m(j_\lambda^*\tilde\calF)$)
is pointwise pure of weight $m$.
An object is called very pure is it is both $!$-pure and $*$-pure.
Note that 
the notion of $!$-purity makes sense for 
ind-objects.
Any very pure object in
$D^b_{\check G(\calO)^{\ct}}(\Gr^{\ct})$
is isomorphic to direct sum of shifted  $\IC_{[\mu]}[m]$. 
We call an ind-object  very pure is it is isomorphic to 
a countable direct sum of shifted $\IC$-complexes.

We have the following important formality properties of $!$-pure objects, extending \cite[Theorem 1.2.2]{G2}
to the twisted setting. To state the results,
it will be convenient to  consider the following variant of the twisted derived Satake equivalences
\[\Psi_{\ct}:\on{Ind}(D_{\check G(\calO)^{\ct}}(\Gr^{\ct}))\cong D^G(\Sym(\fh[-2]))\ \ \ \calF\to \on{R\Gamma}_{\check G(\calO)^{\ct}}(\Gr^{\ct},\calF\otimes^!\bC[H]).\]
Here we viewed $\bC[H]$
as object in $\on{Ind}(D_{\check G(\calO)^{\ct}}(\Gr^{\ct}))$
via the twisted Satake equivalence.
The same argument as in \cite[Section 5(vi)]{BFN2} shows that 
$\Psi_{\ct}\cong (\on{Ad}_{n_0}\circ\on{Chev})_*\circ\Phi_{\ct}$
where $(\on{Ad}_{n_0}\circ\on{Chev})_*$
is the equivalence induced by the 
 involution $\on{Ad}_{n_0}\circ\on{Chev}$
 in Section \ref{symemtric pair}.\footnote{Note that $H$ is stable under $\on{Ad}_{n_0}\circ\on{Chev}$.}
\begin{thm}\label{formality}
Let $\calE\in\on{Ind}(D^b_{\check G(\calO)^{\ct}}(\Gr^{\ct}))$ be a $!$-pure object. 
Then the dg modules 
\[\Psi_{\ct}(\calE)=\on{R\Gamma}_{\check G(\calO)^{\ct}}(\Gr^{\ct},\calE\otimes^!\bC[H])\in D^H(\Sym(\fh[-2]))\] 
\[\on{R\Gamma}_{\check G(\calO)^{\ct}}(\Gr^{\ct},\calE)\in D^{J_H}(\fc_H)\]
are formal, that is, isomorphic to their cohomology 
$\on{H}^*_{\check G(\calO)^{\ct}}(\Gr^{\ct},\calE)$
and $\on{H}^*_{\check G(\calO)^{\ct}}(\Gr^{\ct},\calE\otimes^!\bC[H])$ with trivial differentials.
Moreover,
there is a canonical isomorphism
\begin{equation}\label{Kostant functor}
    \Psi_{\ct}(\calE)\cong\on{H}^*_{\check G(\calO)^{\ct}}(\Gr^{\ct},\calE\otimes^!\bC[H])\cong H^0(j_*)j^*\on{H}^*_{\check G(\calO)^{\ct}}(\Gr^{\ct},\calE\otimes^!\bC[H])\cong
\end{equation}
\[\cong H^0(j_*)(\kappa_{H}^*)^{-1}(\on{H}^*_{\check G(\calO)^{\ct}}(\Gr^{\ct},\calE))\]
Here $j:\fh^{*,reg}\to\fh^*$ is the open embedding and 
$\kappa_{H}^*:\on{QCoh}^H(\fh^{*,reg})\cong \on{QCoh}^{J_H}(\fc_H)$
is the equivalence induces by the Kostant section $\kappa_H:\fc_H\to\fh^{*,reg}$.

\end{thm}
\begin{proof}
Equipped $\Gr^{\ct}=\bigsqcup_{[\mu]\in X_\bullet(\check T)_{\check\theta}}\check I^{\ct}\cdot s_{[\mu]}$ with the $\check T_{\check H}$-stable Iwahori stratification.
Then $\calE$ is $!$-pure with respect to the Iwahori stratification and 
hence by \cite[Lemma 3.1.5]{BY} the cohomology group 
$\on{H}^i_{\check I^{\ct}}(\Gr^{\ct},\calE)$ is pure of weight $i$. It follows that 
$\on{H}^i_{\check G(\calO)^{\ct}}(\Gr^{\ct},\calE)$ is pure of weight $i$ and the formality of $\on{R\Gamma}_{\check G(\calO)^{\ct}}(\Gr^{\ct},\calE)$ follows from the general results in \cite[Section 6.5]{BF} (see also \cite[Section 6.1]{G2}).
Since the regular sheaf $\bC[H]$
on $\Gr^{\ct}$ is very pure, the $!$-tensor product 
$\calE\otimes^!\bC[H]$ is $!$-pure and the formality of 
$\on{R\Gamma}_{\check G(\calO)^{\ct}}(\Gr^{\ct},\calE\otimes^!\bC[H])$ follows from the above discussion.
The fully-faithful results in \cite[Corollary 4.3.7]{G2} (in the case identify map $X=Y$) implies that 
\[\on{H}^*_{\check G(\calO)^{\ct}}(\Gr^{\ct},\calE\otimes^!\bC[H])\cong
(\on{H}^*_{\check G(\calO)^{\ct}}(\Gr^{\ct},\calE)\otimes_{\bC[\fc_H]}\on{H}^*_{\check G(\calO)^{\ct}}(\Gr^{\ct},\bC[H]))^{J_H}\cong 
(\on{H}^*_{\check G(\calO)^{\ct}}(\Gr^{\ct},\calE)\otimes_{\bC[\fc_H]}\bC[\fc_H\times H])^{J_H}
.\]
The isomorphism~\eqref{Kostant functor} follows from the formula 
$(\kappa_H^*)^{-1}(-)\cong((-)\otimes_{\bC[\fc_H]}\bC[\fc_H\times H])^{J_H}$
(see, e.g., \cite[Proposition 2.3]{Ga})

\end{proof}

\begin{remark}
    Note that Theorem \ref{formality}
in particular implies that $!$-pure object $\calE$  (or rather, its image $\Psi_{\ct}(\calE)$ under derived Satake)
is completely determined by its cohomology $\on{H}^*_{\check G(\calO)^{\ct}}(\Gr^{\ct},\calE)$.
\end{remark}

\section{Relative derived Satake equivalence}\label{TRS}
\subsection{Ring objects
arising from  loop symmetric spaces
}\label{ring objects}
We will call the quotient 
$\check X_{\ct}(F)=\check G(F)/\check G(F)^{\ct}$ 
the loop symmetric space.

\begin{example}\label{list 3}
(1) If $\check\theta_0=\on{id}$, then we have 
$\check X_{\ct}(F)\cong\check G(F)/\check G(F')=\check G(\bC((t)))/\check G(\bC((t^2)))$.

(2) 
Here is a list of $\check X_{\ct}(F)$
for non-trivial $\check\theta_0$ :
\[SL_{n}(F)/SU_{n}(F'),\ Spin_{2n}(F)/Spin_{2n}^{(2)}(F'),\ 
 E_6^{sc}(F)/E^{sc,(2)}_6(F').\]
\end{example}

We denote by 
$D_{\check G(\calO)}(\check X_{\ct}(F))$ the co-complete dg category of $\check G(\calO)$-equivariant constructible complexes on $\check X_{\ct}(F)$. We denote by 
$D_{\check G(\calO)}(\check X_{\ct}(F))^c$ and $D^b_{\check G(\calO)}(\check X_{\ct}(F))$ the 
subcategory of locally compact and compact objects,
and $\on{Ind}(D^b_{\check G(\calO)}(\check X_{\ct}(F)))$ the ind-completion.
We have a natural Hecke action of 
$D_{\check G(\calO)}(\Gr)$ on $D_{\check G(\calO)}(\check X_{\ct}(F))$ given by the convolution
\[\calF\star\calM=m_!(\calF\tilde\boxtimes\calM)\]
where $\calF\tilde\boxtimes\calM\in D_{\check G(\calO)}(\check G(F)\times^{\check G(\calO)}\check X_{\ct}(F))$
is the twisted product of $\calF$ and $\calM$, that is, 
$p^!(\calF\tilde\boxtimes\calM)\cong q^!(\calF\boxtimes\calM)\in D_{\check G(\calO)}(\check G(F)\times^{}\check X_{\ct}(F))$
where $p,q$ are the natural projections.
We have similar defined Hecke actions for other versions of categories of sheaves on loop symmetric spaces.
In \cite{CY2}, we will study in details the geometry of 
spherical and Iwahori orbits on $\check X_{\ct}(F)$
and the category $D_{\check G(\calO)}(\check X_{\ct}(F))$
extending our previous work \cite{CY1}.

Consider the closed $\check G(\calO)$-orbit
$\check X_{\ct}(\calO)_e=\check G(\calO)\cdot e\subset \check X_{\ct}(F)$ through the base point $e\in\check X_{\ct}(F)$. 
Let $\omega_e\in D^b_{\check G(\calO)}(\check X_{\ct}(F))$
be the dualizing sheaf on $\check X_{\ct}(\calO)_e$.
As explained in \cite[Section 8.1]{BZSV},
associated to $\omega_e$, there is a commutative ring object 
\[\calA_{\check G,\ct}=\underline{\on{Hom}}_{\on{Ind}(D^b_{\check G(\calO)}(\Gr))}(\omega_e)\in\on{Ind}(D^b_{\check G(\calO)}(\Gr))\]
as the internal endomorphism of 
$\omega_e$ with respect to the Hecke action. 
We have the following more explicit construction 
of the ring object $\calA_{\check G,\ct}$:
consider the scheme $Z_{\check G,\ct}$,
an analogy of variety of triples in \cite{BFN1},
defined by the 
the following cartesian square
\[\xymatrix{Z_{\check G,\ct}\ar[r]\ar@/^2pc/[rr]^{\pi}\ar[d]&\check G(F)\times^{\check G(\calO)}\check X_{\ct}(\calO)_e\ar[r]_{pr_1}\ar[d]^{a}&\Gr=\check G(F)/\check G(\calO)\\
\check X_{\ct}(\calO)_e\ar[r]^i&\check X_{\ct}(F)&&}\]
where $pr_1$ is the projection map, $a$
is the action map, and $i$ is the embedding.
Then one can check that there is an isomorphism 
$\pi_{*}\omega_{Z_{\check G,\ct}}\cong\calA_{\check G,\ct}$.

It follows from the definition that the image of $\calA_{\check G,\ct}$ under the derived Satake equivalence is the de-equivarintized dg algebra 
\[\Psi(\calA_{\check G,\ct})\cong\on{R\Gamma}_{\check G(\calO)}(\Gr,\calA_{\check G,\ct}\otimes^!\bC[\check G])\cong\on{RHom}_{\on{Ind}(D^b_{\check G(\calO)}(\check X_{\ct}(F)))}(\omega_{e},\omega_{e}\star\bC[G]).\]
We have the following alternative construction of 
$\Psi(\calA_{\ct})$ in terms of ring objects in the  derived Satake category for $\Gr^{\ct}$.
Consider  the functor 
\begin{equation}\label{rho}
    \check\rho^!:\on{Ind}(D^b_{\check G(\calO)}(\Gr))\stackrel{\on{For}}\to 
\on{Ind}(D^b_{\check G(\calO)^{\ct}}(\Gr))\stackrel{\check\rho^!}\to 
\on{Ind}(D^b_{\check G(\calO)^{\ct}}(\Gr^{\ct}))
\end{equation}
induced by the embedding 
$\check\rho:\Gr^{\ct}\to\Gr$.

\begin{lemma}\label{lax monoidal}
    There is a canonical lax monoidal structure on $\check\rho^!$.
In particular, 
$\check\rho^!$ sends ring objects (resp. commutative ring objects)
in $\on{Ind}(D^b_{\check G(\calO)}(\Gr))$
to ring objects (resp. commutative ring objects) in 
$\on{Ind}(D^b_{\check G(\calO)^{\ct}}(\Gr^{\ct}))$.
\end{lemma}
\begin{proof}
Consider the following diagram
\[\xymatrix{\Omega\check G^{\ct}_c\times \Omega \check G_c^{\ct}\ar[r]^{\check\rho\times\check\rho}\ar[d]^m&\Omega \check G_c\times \Omega\check G_c\ar[d]^m\\
\Omega\check G_c^{\ct}\ar[r]^{\check\rho}&\Omega \check G_c}.\]
The adjunction map 
$\on{id}\to m^!m_!$ gives rise to a 
map 
$m_!(\check\rho^!\times\check\rho^!)\to\check\rho^!m_!$ and the lax monoidal strucure
\[\check\rho^!\calF\star\check\rho^!\calF'\cong m_!\rho^!\times\check\rho^!(\calF\boxtimes\calF')\to\check\rho^!m_!(\calF\boxtimes\calF')\cong\check\rho^!(\calF\star\calF')\]
where $\calF,\calF'\in D^b_{G(\calO)}(\Gr)\cong D^b_{G_c}(\Omega\check G_c)$.
\end{proof}

\begin{remark}
    One can also adapt the argument in \cite{M} to
    give an algebraic proof 
    without using loop groups $\Omega\check G_c$ and $\Omega\check G^{\ct}_c$.
\end{remark}

The lemma above implies that $\check\rho^!\bC[G]
\in\on{Ind}(D^b_{\check G(\calO)^{\ct}}(\Gr^{\ct}))$
is a commutative ring object.

\begin{lemma}\label{ext algebra}
There is an isomorphism of commutative dg algebras
\[\Psi(\calA_{\check G,\ct})\cong \on{RHom}_{\on{Ind}(D^b_{\check G(\calO)}(\check X_{\ct}(F)))}(\omega_{e},\omega_{e}\star\bC[G])\cong\on{R\Gamma}_{\check G(\calO)^{\ct}}(\Gr^{\ct},\check\rho^!\bC[G])\]
\end{lemma}
\begin{proof}
Consider following diagram
\[\xymatrix{\check G(\calO)^{\ct}\backslash \check G(F)^{\ct}/\check G(\calO)^{\ct}\ar[r]
\ar@/^2pc/[rr]^{\tilde q}\ar[d]^{i}\ar@/_5pc/[dd]_{\check\rho}&\check G(\calO)^{\ct}\backslash \check G(F)^{\ct}/\check G(F)^{\ct}\ar[r]\ar[d]&\check G(\calO)\backslash \check X_{\ct}(\calO)_e\ar[d]^{\tilde i}\\
\check G(\calO)\backslash \check G(F)/\check G(\calO)^{\ct}\ar[r]\ar[d]^p\ar@/_2pc/[rr]_{q}&\check G(\calO)\backslash\check G(F)/\check G(F)^{\ct}\ar[r]&\check G(\calO)\backslash \check X_{\ct}(F)\\
\check G(\calO)\backslash\check G(F)/\check G(\calO)
}\]
where the squares are cartesian.
Then we have 
\[\Psi(\calA_{\check G,\ct})\cong\on{RHom}_{\on{Ind}(D^b_{\check G(\calO)}(\check X_{\ct}(F)))}(\omega_{e},\omega_{e}\star\bC[G])\is \on{RHom}_{\on{Ind}(D^b_{\check G(\calO)}(\check X_{\ct}(F)))}(\omega_{e},q_*p^!\bC[G])\is\]
\[\is\on{RHom}_{\on{Ind}(D^b_{\check G(\calO)}(\check X_{\ct}(\calO)))}(\omega_{e},\tilde i^!q_*p^!\bC[G])
\is
\on{RHom}_{\on{Ind}(D^b_{\check G(\calO)}(\check X_{\ct}(\calO)))}(\omega_{e},\tilde q_*\check\rho^!\bC[G])\is
\]
\[
\is \on{R\Gamma}_{\check G(\calO)^{\ct}}(\Gr^{\ct},\check\rho^!\bC[G])\]
    
\end{proof}

Note that $\bC[G]$ is a very pure object.
It follows that 
 $\check\rho^!\bC[G]$ is 
$!$-pure, and 
Theorem \ref{formality} implies that the commutative dg algebra 
$\on{R\Gamma}_{\check G(\calO)^{\ct}}(\Gr^{\ct},\check\rho^!\bC[G])$ is formal.
Together with Lemma \ref{ext algebra}, we see that $\Psi(\calA_{\check G,\ct})$ is formal, and there is an isomorphism of commutative
algebras
\[\Psi(\calA_{\check G,\ct})\cong\on{H}^*_{\check G(\calO)}(\Gr,\calA_{\check G,\ct}\otimes^!\bC[\check G])\cong\on{H}^*_{\check G(\calO)^{\ct}}(\Gr^{\ct},\check\rho^!\bC[G]).\]

\begin{definition}\label{dual variety}
Following \cite{BFN2,BZSV}, we 
define the relative dual of 
the loop symmetric space $\check X_{\ct}(F)=\check G(F)/\check G(F)^{\ct}$ to be the affine scheme 
$M_{\check G,\check X}$ with coordinate ring
\[\bC[M_{\check G,\ct}]=\Psi(\calA_{\check G,\ct})\cong\on{H}^*_{\check G(\calO)}(\Gr,\calA_{\check G,\ct}\otimes^!\bC[\check G])\cong\on{H}^*_{\check G(\calO)^{\ct}}(\Gr^{\ct},\check\rho^!\bC[G]).\]

\end{definition}

Our next goal to compute 
the coordinate ring
$\bC[M_{\check G,\ct}]$.
To this end, we will apply the 
results of Ginzburg \cite[Theorem 1.1.1 and Corollary 4.3.7]{G2} to 
compute 
$\on{H}^*_{\check G(\calO)^{\ct}}(\Gr^{\ct},\check\rho^!\calF)$
for any very pure object $\calF\in\on{Ind}(D^b_{\check G(\calO)}(\Gr))$.
Here is the key observation: an analogy of  \cite[Lemma 6.3.1]{G2}
for the embedding 
$\check\rho:\Gr^{\ct}\to\Gr$:
\begin{lemma}\label{S1-S3}
 Consider the $\check T_{\check H}$-stable 
 stratification 
 $\Gr=\bigsqcup_{\lambda\in X_\bullet(\check T)}\check I\cdot t^\lambda$ given by Iwahori orbits. We have

 (1) 
 For any $\lambda\in X_\bullet(\check T)$, we have $\Gr^{\ct}\cap\check\rho^{-1}(\check I\cdot t^\lambda)=\check I^{\ct}\cdot s_{[\mu]}$ if $\lambda=\mu+\check\theta_0(\mu)$
 for some $[\mu]\in X_\bullet(\check T)_{\check\theta_0}$ (see~\eqref{Kottwitz})
 , otherwise $\Gr^{\ct}\cap\check\rho^{-1}(\check I\cdot t^\lambda)=\emptyset$.

 (2) Each Iwahori orbit $I\cdot t^\lambda$ is 
 $\check T_{\check H}$-equivariantly
 isomorphic to a vector space and 
 for any $[\mu]\in X_\bullet(\check T)$, the image of the embedding $\check\rho|_{\check I^{\ct}\cdot s_{[\mu]}}:\check I^{\ct}\cdot s_{[\mu]}\to \check I\cdot t^\lambda$
 is a vector subspace.

 (3) We have $\Gr^{\check T_{\check H}}=\bigsqcup_{\lambda\in X_\bullet(\check T)} \{t^\lambda\}$

 (4) 
For every $\lambda\in X_\bullet(\check T)$, 
 there is a $\check T_{\check H}$-stable 
 Zariski open subset $U_\lambda\subset\Gr$ with a $\mathbb G_m$-action that commutes with the $\check T_{\check H}$-action and contracts $U_\lambda$ to $t^\lambda$.
\end{lemma}
\begin{proof}
    (1) $\Gr^{\ct}\cap\check\rho^{-1}(\check I\cdot t^\lambda)$
    is $\check I^{\ct}$-stable, hence a union of $\check I^{\ct}$-orbits on $\Gr^{\ct}$.
    If  $\check I^{\ct}\cdot s_{[\mu]}\subset\Gr^{\ct}\cap\check\rho^{-1}(\check I\cdot t^\lambda)$
    be such an orbit, then 
    $\check\rho(s_{[\mu]})=t^{\mu+\check\theta_0(\mu)}\in\check I\cdot t^\lambda$ and this force $\lambda=\mu+\check\theta_0(\mu)$. Part (1) follows.

    (2) Let $\check I^u$ be the pro-unipotent radical of $\check I$
    and $\check I^u_\lambda$ the stabilizer of $t^\lambda$ in $\check I^u$.
    Then we have a $\check T_{\check H}$-equivariant isomorphism 
    $\Lie(\check I^u/\check I_\lambda^u)\cong\check I\cdot t^\lambda$, $v\to \on{exp}(v)\cdot t^\lambda$. Moreover, the image 
    $\check\rho(\check I^{\ct}\cdot s_{[\mu]})\subset\check I\cdot t^\lambda\cong\Lie(\check I^u/\check I_\lambda^u)$ is equal to the subspace $\Lie(\check I^u/\check I_\lambda^u)^{\ct}$. Part (2) follows.

    (3) It follows from the fact that 
$\check T_{\check H}$ contains regular semisimple elements of $\check G$, see, e.g., \cite[Proposition 6.60]{Kn}.

(4) It follows from existence of big open cell in $\Gr$ (see, e.g., \cite[Lemma 5.8]{Z}).

\end{proof}

\begin{thm}\label{full-faithfulness}
 Let $\calF\in\on{Ind}(D^b_{\check G(\calO)}(\Gr))$ be a very pure object and $\calE\in\on{Ind}(D^b_{\check G(\calO)^{\ct}}(\Gr^{\ct}))$ be a $!$-pure object.
Then there is a natural isomorphism
\begin{equation}\label{twisted Ginzburg}
    \on{H}^*_{\check G(\calO)^{\ct}}(\Gr^{\ct},\calE\otimes^!\check\rho^!\calF)
    \cong (\on{H}^*_{\check G(\calO)^{\ct}}(\Gr^{\ct},\calE)\otimes_{\bC[\fc_H]}\on{H}^*_{\check G(\calO)^{\ct}}(\Gr,\calF))^{J_G|_{\fc_H}}.
\end{equation}
Here $J_G|_{\fc_H}$ acts on 
the tensor product through the natural action on $\on{H}^*_{\check G(\calO)^{\ct}}(\Gr,\calF)$
and the pull-back action of $J_H$ on $\on{H}^*_{\check G(\calO)^{\ct}}(\Gr^{\ct},\calE)$ via the norm map
$\on{Nm}:J_G|_{\fc_H}\to J_H$.
If $\calF$ and $\calE$ are ring objects then~\eqref{twisted Ginzburg} is a ring isomorphism.

\end{thm}
\begin{proof}

 Lemma \ref{S1-S3} implies that 
the $\check T_{\check H}$-stable 
 stratification 
 $\Gr=\bigsqcup_{\lambda\in X_\bullet(\check T)}\check I\cdot t^\lambda$ satisfies the condition (S1-S3)
 in \cite[Section 1.1]{G2}, and the claim  follows \cite[Corollary 4.3.7]{G2} 
 in the case of the embedding 
$\check\rho:\Gr^{\ct}\to\Gr$ and the formula 
in Proposition \ref{cohomology} (2)
\end{proof}

\begin{remark}
In fact, since the $\check T_{\check H}$ fixed points $\Gr^{\check T_{\check H}}=\bigsqcup_{\lambda\in X_\bullet(\check T)} \{t^\lambda\}$ are isolated, one can follow the strategy in \cite{G1,BY,CMNO} to give an alternative proof of Theorem \ref{full-faithfulness} (compare \cite[Remark 1.1.3]{G2}).
\end{remark}

\begin{thm}\label{dual construction}

(1) The complex  $\check\rho^!\bC[G]$ is equivariantly formal, that is, the spectral sequence
\[\on{H}^p_{\check G(\calO)^{\ct}}(\on{pt},\on{H}^q(\Gr^{\ct},\rho^!\bC[G]))\Longrightarrow\on{H}^{p+q}_{\check G(\calO)^{\ct}}(\Gr^{\ct},\rho^!\bC[G])\]
degenerate at $E_2$.

(2)
There is a  $G$-equivariant isomorphism
\[\bC[M_{\check G,\ct}]\cong\Psi(\calA_{\ct})\cong\on{H}^*_{\check G(\calO)^{\ct}}(\Gr^{\ct},\check\rho^!\bC[G])\cong \bC[T^*X[2]]\]
of $\Sym(\fg[-2])$-algebras.

(3)
There is a  $G$-equivariant isomorphism
\[
\on{H}^*_{}(\Gr^{\ct},\check\rho^!\bC[G])\cong \bC[G\times^K\calN_{\fp^*}].\]

\end{thm}
\begin{proof}
(1) 
Note that the parity vanishing of  $\bC[G]$ implies $H^j(i^!_{[\mu]}\check\rho^!\bC[G])\cong H^j(i^!_\lambda\bC[G])=0$
for $j$ not congruent to $\langle\mu,2\rho_G\rangle=\on{dim}(\check G(\calO)^{\ct}\cdot s_{[\mu]})$.
Since $\check\rho^!\bC[G]$ is constructible with 
the Iwahori stratification, 
the claim  follows from the  observation in \cite[Proposition 2.1.1 and Remark 2.1.2]{M}.

(2) Applying Theorem \ref{full-faithfulness} to the case $\calE=\omega$ and $\calF=\bC[G]$, we obtain
\[
\bC[M_{\check G,\ct}]\cong\on{H}^*_{\check G(\calO)^{\ct}}(\Gr^{\ct},\check\rho^!\bC[G])\cong
(\on{H}^*_{\check G(\calO)^{\ct}}(\Gr^{\ct},\omega)\otimes_{\bC[\fc_H]}\on{H}^*_{\check G(\calO)^{\ct}}(\Gr,\bC[G]))^{J_G|_{\fc_H}}\cong\]
\[\cong\bC[J_H\times_{\fc_H}G\times\fc_H]^{J_G|_{\fc_H}}
\cong\bC[T^*X]\]
where the last isomorphism is Theorem \ref{Wh=T^*X}.

(3) 
Note that the quotient map $T^*X\to T^*X/G\cong (G\times^K\fp^*)/G\cong\fp^*/K\cong\fc_X$
is faithfully flat and 
$T^*X\times_{\fc_X}\{0\}\cong G\times^K\calN_{\fp^*}$.
Now part (1) and (2) implies
\[\on{H}^*_{}(\Gr^{\ct},\check\rho^!\bC[G])\cong 
\on{H}^*_{\check G(\calO)^{\ct}}(\Gr^{\ct},\check\rho^!\bC[G])\otimes_{\on{H}^*_{\check G(\calO)^{\ct}}(\on{pt})}\bC\cong\bC[T^*X\times_{\fc_X}\{0\}]\cong\bC[G\times^K\calN_{\fp^*}].\]

\end{proof}

\begin{remark}
Since the symmetric nilptent cone $\calN_{\fp^*}$
in general is not normal, e.g., the case $X=PGL_2/PO_2$, we have 
$\calN_{\fp^*}\cong\{xy=0\}\subset\bC^2$, it follows that 
$\on{H}^*_{}(\Gr^{\ct},\check\rho^!\bC[G])\cong\bC[G\times^K\calN_{\fp^*}]$ is not a normal ring in general.
This also implies that the analogy of Theorem \ref{full-faithfulness} is not true in the non-equivariant setting
(This is compatible with the comments after \cite[Corollary 1.1.2]{G2}).
Indeed, in the case
above the non-equivariant analogy of the right hand side of \eqref{twisted Ginzburg} will give you
the ring of functions $\bC[\calN_{\fp^{*,r}}] \cong\bC[x^{\pm1}]\times\bC[y^{\pm1}]$
on the regular nilpotent orbit $\calN_{\fp^{*,r}}=\calN_{\fp^*}\setminus\{0\}$, which is not the same as $\bC[\calN_{\fp^*}]$.
\end{remark}

\subsection{Functorial kernels}
We shall give a spectral description 
of the functor 
$\check\rho^!$ under the derived Satake equivalence.
Consider the image 
$\Psi_{\ct}(\check\rho^!\bC[G])
=\on{R\Gamma}_{\check G(\calO)^{\ct}}(\Gr^{\ct},\check\rho^!\bC[G]\otimes^!\bC[H])\in D^H(\Sym(h[-2]))$ of $\check\rho^!\bC[G]$
under the derived Satake equivalence.
Recall the scheme 
$T^*X\times H/J_H$ 
in Section \ref{contangent bundle}.

\begin{thm}\label{functoriality}
(1) $\Psi_{\ct}(\check\rho^!\bC[G])$ is formal and there is a $G\times H$-equivariant isomorphism 
\[\Psi_{\ct}(\check\rho^!\bC[G])\cong\bC[T^*X[2]\times H/J_H]\]
of $\Sym((\fg\times\fh)[-2])$-algebras.

 (2) The following diagram commutes:
    \[\xymatrix{\on{Ind}(D^b_{\check G(\calO)}(\Gr))\ar[r]^{\check\rho^!}\ar[d]^{\simeq}&\on{Ind}(D^b_{\check G(\calO)^{\ct}}(\Gr^{\ct}))\ar[d]^{\simeq}\\
D^G(\Sym(\fg[-2]))\ar[r]&D^H(\Sym(\fh[-2]))}\]
where the lower horizontal arrow is the functor
\[(k[(T^*X[2]\times H)/J_H]\otimes_{\Sym(\fg[-2])}(-))^G:D^G(\Sym(\fg[-2]))\to D^H(\Sym(\fh[-2]))\]
\end{thm}
\begin{proof}
(1) 
Note that 
 $\bC[H]$ and $\bC[G]$ is very pure 
and 
 $\check\rho^!\bC[G]$ is  $!$-pure  
and hence  Theorem \ref{full-faithfulness} implies 
$\Psi_{\ct}(\check\rho^!\bC[G])\cong\on{R\Gamma}_{\check G(\calO)^{\ct}}(\Gr^{\ct},\check\rho^!\bC[G]\otimes^!\bC[H])$ is formal and there is an isomorphism
\begin{equation}\label{part 1}
    \Psi_{\ct}(\check\rho^!\bC[G])\cong\on{H}^*_{\check G(\calO)^{\ct}}(\Gr^{\ct},\check\rho^!\bC[G]\otimes^!\bC[H])
    \cong (\on{H}^*_{\check G(\calO)^{\ct}}(\Gr^{\ct},\bC[H])\otimes_{\bC[\fc_H]}\on{H}^*_{\check G(\calO)^{\ct}}(\Gr,\bC[G]))^{J_G|_{\fc_H}}\cong
\end{equation}
    \[\cong\bC[G\times H\times\fc]^{J_G|_{\fc_H}}\cong\bC[T^*X\times H/J_H]^{}.\]
 where  the last isomorphism is Corollary \ref{kernel}. 

(2)
We shall
construct a functorial isomorphism 
\[\Psi_{\ct}\circ\check\rho^!\circ\Psi^{-1}(V)\cong (k[T^*X\times H/J_H]\otimes_{\Sym(\fg[-2])}(V))^G\]
for any $V\in D^G(\Sym(\fg[-2]))$.
The same argument as in the proof of Proposition 
\ref{Comp} shows that it suffices to construct such isomorphism when $V=\Sym(\fg[-2])\otimes\bC[G])$.
But it follows from the following canonical isomorphisms
 \[(k[T^*X\times H/J_H]\otimes_{\Sym(\fg[-2])}(\Sym(\fg[-2])\otimes\bC[G]))^G\cong
k[T^*X\times H/J_H]\otimes\bC[G]))^G\cong k[T^*X\times H/J_H]\]
and 
\[\Psi_{\ct}\circ\check\rho^!\circ\Psi^{-1}(\Sym(\fg[-1])\otimes\bC[G])\cong\Psi_{\ct}(\check\rho^!\bC[G])\cong k[T^*X\times H/J_H]\]
where the last isomorphism is~\eqref{part 1}.

\end{proof}

\subsection{Derived Satake for loop symmetric spaces}\label{coherent sheaves}
Denote by
$D^{K}(\on{Sym}(\fp[-2]))$ be the monoidal dg-category of  $K$-equivariant dg-modules 
over the dg-algebra $\on{Sym}(\fp[-2])$ (equipped with trivial differential) with monoidal structure given by (derived) tensor prodcut.
We denote by 
$D^{K}(\on{Sym}(\fp[-2]))^c=D^{K}_{\on{perf}}(\on{Sym}(\fp[-2]))$
the full subcategory compact objects 
consisting of perfect modules.
Denote by $D^{K}(\on{Sym}(\fp[-2]))_{\mathcal N_{\fp^*}}$
and  $D^{K}_{\on{perf}}(\on{Sym}(\fp[-2]))_{\calN_{\fp^*}}$
the full subcategory of 
consisting of modules that are set theoretically supported on the symmetric nilpotent 
 cone $\calN_{\fp^*}$.
 Note that restriction the 
 zero section
 $\fp^*\to T^*X\cong G\times^K\fp^*$
induces an equivalence 
$\on{QCoh}^G(T^*X[2])\cong D^K_{\on{}}(\Sym(\fp[-2]))$
(resp. $\on{Perf}^G(T^*X[2])\cong D^K_{\on{perf}}(\Sym(\fp[-2]))$)
we denote by 
$\on{QCoh}^G(T^*X[2])_{\calN_{\fp^*}}$
(resp. $\on{Perf}^G(T^*X[2])_{\calN_{\fp^*}}$) the corresponding full subcategories.

We denote by 
$\on{Ind}(D^b_{\check G(\calO)}(\check X_{\ct}(F)))_e$ the full subcategory 
of $\on{Ind}(D^b_{\check G(\calO)}(\check X_{\ct}(F)))$
generated by $\omega_e$ under the Hecke action.
We denote by $D_{\check G(\calO)}(\check X_{\ct}(F)))_e\subset D_{\check G(\calO)}(\check X_{\ct}(F)))$ the full subcategory generated by 
$\tilde\omega_e=p_!\omega_{\on{pt}}$
where $p:\{\on{pt}\}\to \check G(\calO)^{\ct}\backslash\{\on{pt}\}\cong
\check G(\calO)\backslash \check X_{\ct}(\calO)_e$ 
under the Hecke action
(note that $\tilde\omega_e$ is a compact object in $D_{\check G(\calO)}(\check X_{\ct}(F)))$).
We denote by $D_{\check G(\calO)}(\check X_{\ct}(F)))^c_e\subset D_{\check G(\calO)}(\check X_{\ct}(F)))_e$ the full subcategory of compact objects.
We will call the above full subcategories the 
Satake categories for the loop symmetric space $\check X_{\ct}(F)$.

\begin{thm}\label{twisted DSat}
(1) There is a canonical  equivalence 
\[
 \on{Ind}\Upsilon^b_{}:\on{Ind}(D^b_{\check G(\calO)}(\check X_{\ct}(F)))_e\cong \on{QCoh}^G(T^*X[2])\cong D^{K}(\on{Sym}(\fp[-2]))\]
 compatible with the monoidal action of the derived Satake equivalence for $\check G$ in~\eqref{DSat} on both sides.
 Moreover, it 
restricts to an equivalence
\[
 \Upsilon^b_{}:D^b_{\check G(\calO)}(\check X_{\ct}(F))_{e}\cong 
 \on{Perf}^G(T^*X[2])\cong D^{K}_{\on{perf}}(\on{Sym}(\fp[-2]))
 \]
 on compact objects.

 (2) There is a canonical  equivalence 
 \[
\Upsilon_{}:D_{\check G(\calO)}(\check X_{\ct}(F))_e\cong 
\on{QCoh}^G(T^*X[2])_{\calN_{\fp^*}}\cong
D^{K}(\on{Sym}(\fp[-2]))_{\calN_{\fp^*}}
\]
 compatible with the monoidal action of the derived Satake equivalence for $\check G$ in~\eqref{DSat2} on both sides.
 Moreover, it
 restricts to an equivalence
 \[
\Upsilon^c_{}:D_{\check G(\calO)}(\check X_{\ct}(F))_e^c
\cong 
\on{Perf}^G(T^*X[2])_{\calN_{\fp^*}}\cong
D^{K}_{\on{perf}}(\on{Sym}(\fp[-2]))_{\calN_{\fp^*}}\]
on compact objects.
\quash{
(3) There is an isomorphism 
\[\on{H}^*_{\check G(\calO)^{\ct}}(\Gr^{\ct},\calF)\cong\on{Ind}(\Phi^b_{\check\theta})(\calF)\otimes_{\Sym(\fh[-2])}\bC[\fc_H]\] compatible 
with the action of 
$\on{H}^H_*(\Gr^{\check\theta})\cong J_H$.
The map $\Sym(\fh[-2])\cong\bC[\fh^*[-2]]\stackrel{\kappa_H^*}\to\bC[\fc_H]$
is indcued by the Kostant section 
$\kappa_H:\fc_H\to\fh^*$.}
\end{thm}
\begin{proof}
    (1) 
    Recall the 
    de-equivariantized dg algebra 
    $\Psi(\calA_{\ct})\cong\on{RHom}_{\on{Ind}(D^b_{\check G(\calO)}(\check X_{\ct}(F)))}(\omega_{\check X_{\ct}(\calO)},\omega_{\check X_{\ct}(\calO)}\star\bC[G])$ in Section \ref{TRS}.
    The same argument as in Theorem \ref{twisted DSat}
    implies that the functor 
    \[\on{Ind}(D^b_{\check G(\calO)}(\check X_{\ct}(F)))_e\cong D^G(\Psi(\calA_{\ct})),\ \ \ \calF\to \on{RHom}_{\on{Ind}(D^b_{\check G(\calO)}(\check X_{\ct}(F)))}(\omega_{e},\calF\star\bC[G])\]
    is an
   equivalent where
    $D^G(\Psi(\calA_{\ct}))$ 
    is the category 
    of $G$-equivariant dg modules over $\Psi(\calA_{\ct})$. 
    By Theorem \ref{dual construction}, we have 
    $\Sym(\fg[-2])$-algebras isomorphism $\Psi(\calA_{\ct})\cong\bC[T^*X[2]]$ and hence \[\on{Ind}(D^b_{\check G(\calO)}(\check X_{\ct}(F)))_e\cong D^G(\Psi(\calA_{\ct}))\cong \on{QCoh}^G(T^*X[2])\cong D^K(\Sym(\fp[-2]))\]
    Part (1) follows.

    (2) We claim that under the fully-faithful embedding
    \[D_{\check G(\calO)}(\check X_{\ct}(F))_e\cong\on{Ind}(D_{\check G(\calO)}(\check X_{\ct}(F))_e)^c\to\on{Ind}(D^b_{\check G(\calO)}(\check X_{\ct}(F))_e)\cong \on{QCoh}^G(T^*X[2])\]
    the object $\tilde\omega_e$
    goes to 
    $\bC[G\times^K\calN_{\fp^*}]$.
    We conclude that $D_{\check G(\calO)}(\check X_{\ct}(F))_e$
is equivalent to the full subcategory of $ \on{QCoh}^G(T^*X[2])$ generated by
objects of the form 
$\bC[G\times^K\calN_{\fp^*}]\otimes V$
for some $V\in\on{Rep}(G)$. It is clear that this subcategory is exactly $\on{QCoh}^G(T^*X[2])_{\calN_{\fp^*}}$.

To prove the claim, we note that 
    the image of $\tilde\omega_e$ is given by 
    $\on{RHom}_{\on{Ind}(D^b_{\check G(\calO)}(\check X_{\ct}(F)))}(\omega_{e},\tilde\omega_{e}\star\bC[G])$ and using the diagram in Lemma \ref{ext algebra} one can show that 
    \[\on{RHom}_{\on{Ind}(D^b_{\check G(\calO)}(\check X_{\ct}(F)))}(\omega_{e},\tilde\omega_{e}\star\bC[G])\cong\on{R\Gamma}(\Gr^{\ct},\check\rho^!\bC[G]).\]
  Now the formality property in Theorem \ref{formality} and the computation in Theorem \ref{ext algebra} (3) implies the desired claim
  \[\on{R\Gamma}(\Gr^{\ct},\check\rho^!\bC[G])\cong \on{H}^*(\Gr^{\ct},\check\rho^!\bC[G])\cong\bC[G\times^K\calN_{\fp^*}].\]
   \end{proof}

\subsection{Bezrukavnikov equivalence for quasi-split symmetric spaces}\label{Bez}
Recall the Grothendieck-Springer resolution $\widetilde\fg\to\fg^*$
 parametrizing pairs
$(v,B')$ where $B'$ is a Borel subgroup of $G$ and $v\in\fg^*$ is a
linear map which vanishes on the Lie algebra of unipotent radical of $B'$,
and the map to $\fg^*$ is the natural projection map. 
We have a natural $G$-action on $\widetilde\fg\to\fg^*$ given by $g(v,B')=(\on{Ad}_g(v),\on{Ad}_g(B'))$.
Consider the base change $\fp^*\times_{\fg^*}\widetilde\fg$
along the embedding $\fp^*\to\fg^*$.
Note that $\fp^*\times_{\fg^*}\widetilde\fg$
is stable under the action of the symmetric subgroup $K$. Moreover,
there is an isomorphism of stacks
\[(T^*X\times_{\fg^*}\widetilde\fg)/G\cong(\fp^*\times_{\fg^*}\widetilde\fg)/K.\]
Let $\check I\subset\check G(\calO)$ be the Iwahori subgroup associated to the Borel $\check B$.
In \cite{CY2}, we extend our 
 previous work 
 \cite[Theorem 36 and Proposition 4.1]{CY1} on placidness of $\check X(F)$ to the setting of loop symmetric space $\check X_{\ct}(F)$ and 
we show 
that 
$\check X_{\ct}(F)$ admits a dimension theory and the 
colimit of 
$\check G(\calO)$-orbit closures (resp. $\check I$-orbits closures)
gives rise to a placid presentation of $\check X_{\ct}(F)$.
Let $D^b_{\check I}(\check X_{\ct}(F))$
be the $\check I$-equivariant derived category of $\check X_{\ct}(F)$. We have a natural monoidal action of the affine Hecke category 
$D^b_{\check I}(\check G(F)/\check I)$
on $D^b_{\check I}(\check X_{\ct}(F))$ given 
by the convolution product 
and we denote by 
$D^b_{\check I}(\check X_{\ct}(F))_e\subset D^b_{\check I}(\check X_{\ct}(F))$ 
the full subcategory generated by $\omega_{e}$, viewed as an object in $D^b_{\check I}(\check X_{\ct}(F))$, under the
affine Hecke action of $D^b_{\check I}(\check G(F)/\check I)$.
Now the derived Satake equivalence 
for loop symmetric spaces 
in Theorem \ref{twisted DSat} together with \cite[Theorem 5.1]{LPY} imply the following 
constructible realization of 
equivariant perfect complexes on 
$T^*X\times_{\fg^*}\widetilde\fg$ and $\fp^*\times_{\fg^*}\widetilde\fg$:

\begin{thm}\label{Bez equivalence}
 There are natural equivalences
 \[D^b_{\check I}(\check X_{\ct}(F))_e\cong
 \on{Perf}^G((T^*X\times_{\fg^*}\widetilde\fg)[2])\cong 
 \on{Perf}^K((\fp^*\times_{\fg^*}\widetilde\fg)[2]) \]
\end{thm}

The relationship between the equivalences above with \cite{B} is as follows.
Consider the Springer resolution 
$\widetilde\calN\to\fg$ 
 parametrizing pairs
$(v,B')$ where $B'$ is a Borel subgroup of $G$ and $v\in\fg$ is a vector lies in the Lie algebra of unipotent radical of $B'$,
and the map to $\fg$ is the natural projection map. Then there is a  Koszul duality equivalence
\[
\on{Perf}^K((\fp^*\times_{\fg^*}\widetilde\fg)[2])\cong
\on{Coh}^K((\fk\times_{\fg}\widetilde\calN)[-1])\]
and the composed equivalence
\[D^b_{\check I}(\check X_{\ct}(F))_e\cong \on{Perf}^K((\fp^*\times_{\fg^*}\widetilde\fg)[2])\cong\on{Coh}^K((\fk\times_{\fg}\widetilde\calN)[-1])\]
provides a generalization of 
\cite[Theorem 1(4)]{B} to quasi-split symmetric spaces.\footnote{We thank S. Devalapurkar
for explaining this fact to the author.} 
 
\begin{remark}
It would be very interesting to study the properties of the equivalences in Theorem \ref{Bez equivalence}, such as compatibility with $t$-structures, like the case in \cite{B}.
\end{remark}

\subsection{Variants: triality, special parahoric, and subfield $F_n=\bC((t^n))$ }\label{variant}
We discuss generalizations 
of the our main results in the previous sections 
to some non-symmetric spaces settings. 
Many results in this section were 
also discovered independently by  S. Devalapurkar \cite[Section 4]{De1} and \cite[Section 5.5.6]{De2}.

\subsubsection{Triality of $Spin_{8}$}
Consider the case $G=PSO_8$
and $\check G=Spin_{8}$. Let 
$\theta_0\in\on{Aut}(PSO_8)$ and  $\check\theta_0\in\on{Aut}(Spin_8)$ be the unique   order $3$ pinned automorphism (i.e. the 
triality).
Let $G^{\theta_0}\cong\check G^{\check\theta_0}\cong G_2$
be the fixed points subgroups.
Fix a primitive $3$-th root $\mu_3$ and let $\ct:Spin_{8}(F)\to Spin_{8}(F)$ is be order three automorphism $\ct(\gamma)(t)=\check\theta_0(\gamma(\mu_3t))$.
Let $Spin_{8}(F)^{\ct}$ and $Spin_{8}(\calO)^{\ct}$
be the corresponding twisted loop and arc group.

Since $\theta_0$ is a pinned automorphism, 
the fixed point $\fg_2^*=(\frak{so}_8^*)^{\theta_0}$
contains a principal $\frak{sl}_2$ of $\frak{so}_8$,
and it follows that $\fg_2^{*,r}=\fg_2^*\cap\frak{so}_8^{*,r}$ and hence the same discussion in Section \ref{involutions}
shows that
the natural automorphism on
$I_{\frak{so}_8}|_{\fg_2^{*,r}}$ sending $(g,h)\to (\theta_0(g),h)$ descends to an automorphism 
$\theta_{0,J}$ on $J_{PSO_8}|_{\fc_{G_2}}$ with fixed points $(J_{PSO_8}|_{\fc_{G_2}})^{\theta_{0,J}}\cong\ J_{G_2}$.
We define $\on{Nm}:J_{PSO_8}|_{\fc_{G_2}}\to J_{G_2}$
the norm map sending $g$ to $\on{Nm}(g)=g\theta_{0,J}(g)\theta^2_{0,J}(g)$.
Note that $Spin_8(F)^{\ct}/Spin_8(\calO)^{\ct}$
is an example of twisted affine Grassmannian in considered in \cite{Z}. In particular,
we have the geometric Satake equivalence
\[\on{Perv}_{Spin_8(\calO)^{\ct}}(Spin_8(F)^{\ct}/Spin_8(\calO)^{\ct})\cong\on{Rep}(G_2)\]
and the $\IC$-compelxes in $\on{Perv}_{Spin_8(\calO)^{\ct}}(Spin_8(F)^{\ct}/Spin_8(\calO)^{\ct})$
 are very pure and satisfy the parity vanishing condition.

Consider the affine scheme
\[\overline{PSO_8\times J_{G_2}/J_{PSO_8}|_{\fc_{G_2}}}=\on{Spec}(\bC[PSO_8\times J_{G_2}]^{J_{PSO_8}|_{\fc_{G_2}}})\]
where $J_{PSO_8}|_{\fc_{G_2}}$ acts on 
$PSO_8$ via the embedding 
$J_{PSO_8}|_{\fc_{G_2}}\to\fc_{G_2}\times PSO_8$ induced by a Kostant section 
and on $J_{G_2}$ via the norm map.

\begin{thm}\label{Triality}
    (1) There is an isomorphism
\[\on{Spec}(\on{H}_*^{G_2
}(Spin_8(F)^{\ct}/Spin_8(\calO)^{\ct}))\cong J_{G_2}\]
of group schemes  fitting into the following diagram commutes
    \[\xymatrix{\on{Spec}(\on{H}_*^{G_2}(Spin_8(F)/Spin_8(\calO)))\ar[r]^{ \rho^{G_2}}\ar[d]^\simeq&\on{Spec}(\on{H}_*^{G_2}(Spin_8(F)^{\ct}/Spin_8(\calO)^{\ct})))\ar[d]^\simeq\\
J_{PSO_8}|_{\fc_{G_2}}\ar[r]^{\on{Nm}}&J_{G_2}}\]
    where $\on{Nm}$ is the norm map
    and the map $\rho^{G_2}_*$
    is induced by the embedding 
    $\check\rho:Spin_{8}(F)^{\ct}/Spin_8(\calO)^{\ct}\to Spin_{8}(F)/Spin_8(\calO)$.

(2) There is a monoidal equivalence 
\[D^b_{Spin_8(\calO)^{\ct}}(Spin_8(F)^{\ct}/Spin_8(\calO)^{\ct})\cong\on{Perf}^{G_2\times G_2}(T^*G_2[2])\]

(3) There is an equivalence 
\[D^b_{Spin_8(\calO)}(Spin_8(F)/Spin_8(F)^{\ct})_e\cong \on{Perf}^{PSO_8}_{}(\overline{PSO_8\times J_{G_2}/J_{PSO_8}|_{\fc_{G_2}}}).\]
Here $D^b_{Spin_8(\calO)}(Spin_8(F)/Spin_8(F)^{\ct})_e\subset D^b_{Spin_8(\calO)}(Spin_8(F)/Spin_8(F)^{\ct})$ is full subcategory 
generate by the dualizing sheaf $\omega_e$
of the $Spin_8(\calO)$-orbit through the base point $e\in Spin_8(F)/Spin_8(F)^{\ct}$
under the Hecke action.

    \end{thm}
\begin{proof}
Since 
the $\check T^{\ct_0}=\check T_{G_2}$-fixed points on  
$Spin_8(F)^{\ct}/Spin_8(\calO)^{\ct}$ are 
parametrized by 
$ X_\bullet(\check T_{})_{\check\theta}\cong \check T_{}(F)^{\check\theta}/\check T(\calO)^{\check\theta}$ with embedding 
$ X_\bullet(\check T)_{\check\theta}\cong\check T(F)^{\check\theta}/T(\calO)^{\check\theta}\to \check T(F)^{}/T(\calO)^{}$
given by
$[\mu]\to t^{\mu+\check\theta_0(\mu)+\check\theta_0^2(\mu)}\cdot\check T(\calO)$. Now we can apply the same localization argument in Proposition \ref{cohomology} to prove (1).

Since $G_2$ contains regular semisimple elements of $Spin_8$,  Lemma \ref{S1-S3}
holds true for the embedding $\check\rho$,
the same argument as in the symmetric spaces setting implies (2) and (3).

\end{proof}

\subsubsection{Special parahoric of $SU_{2n+1}(F')$}\label{Special}
Let $G=PGL_{2n+1}$
and $\check G=SL_{2n+1}$, $n\geq 1$.
Let $\ct_0\in\on{Aut}(\check G,\check B,\check T,\check x)$  be the non-trivial pinned automorphism of $SL_{2n+1}$  with $(SL_{2n+1})^{\ct_0}\cong SO_{2n+1}$.
Let $\theta_0$ be corresponding pinned automorphism of $PGL_{2n+1}$ with 
fixed points $(PGL_{2n+1})^{\theta_0}\cong SO_{2n+1}$
Fix a primitive $4$-th root 
$\mu_4$ of unity.
Consider the field extension $E=\bC((t^{1/2}))\supset F$
and its ring of integers $\calO_E=\bC[[t^{1/2}]]$.
According to \cite[Section 2.1]{BH}
and \cite[Theorem 4.1.2]{DH}
, which goes back to the work of Kac \cite[Section 8]{Kac}, there is 
an
order $4$ automorphism $\sigma_0\in\on{Aut}(SL_{2n+1})$ 
satisfying the following properties:

(1) We have $\sigma_0=\on{Ad}_{s}\circ\ct_0$
for some $s\in\check T^{\check\theta_0}$ with $s^4=1$

(2) We have $(SL_{2n+1})^{\sigma_0}\cong Sp_{2n}$, 

(3) Let $\sigma:SL_{2n+1}(E)\to SL_{2n+1}(E)$ be the order $4$
automorphism of 
$SL_{2n+1}(E)$ given by 
$\sigma(\gamma)(t^{1/2})=\sigma_0(\gamma)(\mu_4^{-1}t^{1/2})$.
Then there is an isomorphism 
\[SL_{2n+1}(E)^{\sigma}\cong SL_{2n+1}(F)^{\ct}\cong SU_{2n+1}(F')\] such such the subgroup $SL_{2n+1}(\calO_E)^{\sigma}\subset SL_{2n+1}(E)^{\sigma}$ goes to the 
special parahoric 
of $P\subset SU_{2n+1}(F')$ with reductive quotient 
isomorphic to $Sp_{2n}$ (see \cite[Section 8]{Z}). 
\\

It follows that 
the quotient 
$SL_{2n+1}(E)^{\sigma}/SL_{2n+1}(E)^{\sigma}\cong SL_2(F)^{\ct}/P$
is an example of the twisted affine Grassmannians considered in \cite{Z}. In particular, we have the geometric Satake quivalence 
\[\on{Perv}_{SL_{2n+1}(E)^{\sigma}}(SL_{2n+1}(E)^{\sigma}/SL_{2n+1}(E)^{\sigma})\cong 
\on{Perv}_{P}(SU_{2n+1}(F')/P)\cong\on{Rep}(SO_{2n+1})\]
and the $\IC$-compelxes in $\on{Perv}_{SL_{2n+1}(E)^{\sigma}}(SL_{2n+1}(E)^{\sigma}/SL_{2n+1}(E)^{\sigma})$
 are very pure and satisfy the parity vanishing condition.
Consider the affine scheme
\[\overline{(PGL_{2n+1}\times J_{SO_{2n+1}}/J_{PGL_{2n+1}}|_{\fc_{SO_{2n+1}}}})=\on{Spec}(\bC[PGL_{2n+1}\times J_{SO_{2n+1}}]^{J_{PGL_{2n+1}}|_{\fc_{SO_{2n+1}}}})\]
where $J_{PGL_{2n+1}}|_{\fc_{SO_{2n+1}}}$ acts on  $J_{SO_{2n+1}}$ via the norm map
$\on{Nm}:J_{PGL_{2n+1}}|_{\fc_{SO_{2n+1}}}\to J_{SO_{2n+1}}$
in~\eqref{norm}.

\begin{thm}
    (1) There is an isomorphism
\[\on{Spec}(\on{H}_*^{Sp_{2n}
}(SL_{2n+1}(E)^{\sigma}/SL_{2n+1}(\calO_E)^{\sigma}))\cong J_{SO_{2n+1}}\]
of group schemes over $\check\fc_{Sp_{2n} }\cong\fc_{SO_{2n+1}}$ fitting into the following diagram commutes
    \[\xymatrix{\on{Spec}(\on{H}_*^{Sp_{2n}}(SL_{2n+1}(E)^{}/SL_{2n+1}(\calO_E)))\ar[r]^{ \rho^{Sp_{2n}}}\ar[d]^\simeq&\on{Spec}(\on{H}_*^{Sp_{2n}}(SL_{2n+1}(E)^{\sigma}/SL_{2n+1}(\calO_E)^{\sigma})))\ar[d]^\simeq\\
J_{PGL_{2n+1}}|_{\fc_{SO_{2n+1}}}\ar[r]^{[2]\circ\on{Nm}}&J_{SO_{2n+1}}}\]
    where 
     the map $\rho^{Sp_{2n}}$
    is induced by the embedding 
    \[\check\rho:SL_{2n+1}(E)^{\sigma}/SL_{2n+1}(\calO_E)^{\sigma}\to SL_{2n+1}(E)^{}/SL_{2n+1}(\calO_E)^{}\]

(2) There is a monoidal equivalence 
\[D^b_{SL_{2n+1}(\calO_E)^{\sigma}}(SL_{2n+1}(E)^{\sigma}/SL_{2n+1}(\calO_E)^{\sigma})\cong \on{Perf}^{SO_{2n+1}\times SO_{2n+1}}(T^*SO_{2n+1}[2])\]

(3) There is an equivalence 
\[D^b_{SL_{2n+1}(\calO_E)^{}}(SL_{2n+1}(E)^{}/SL_{2n+1}(E)^{\sigma})_e\cong \on{Perf}^{PGL_{2n+1}}_{}(\overline{PGL_{2n+1}\times J_{SO_{2n+1}}/J_{PGL_{2n+1}}|_{\fc_{SO_{2n+1}}}}).\]
Here $D^b_{SL_{2n+1}(\calO_E)^{}}(SL_{2n+1}(E)^{}/SL_{2n+1}(E)^{\sigma})_e\subset D^b_{SL_{2n+1}(\calO_E)^{}}(SL_{2n+1}(E)^{}/SL_{2n+1}(E)^{\sigma})$ is full subcategory 
generate by the dualizing sheaf $\omega_e$
of the $SL_{2n+1}(\calO_E)^{}$-orbit through the base point $e\in SL_{2n+1}(E)^{}/SL_{2n+1}(E)^{\sigma}$
under the Hecke action.
    
    \end{thm}
\begin{proof}
 The $\check T^{\sigma_0}=T_{Sp_{2n}}$-fixed points on  
$SL_{2n+1}(E)^{\sigma}/SL_{2n+1}(\calO_E)^{\sigma}$ are 
parametrized by 
$ X_\bullet(\check T_{})_{\sigma}\cong \check T_{}(E)^{\sigma}/\check T(\calO_E)^{\sigma}$ with embedding 
$ X_\bullet(\check T)_{\sigma}\cong\check T(E)^{\sigma}/\check T(\calO_E)^{\sigma}\to \check T(E)^{}/T(\calO_E)^{}$
given by
\[[\mu]\to {(t^{1/2})}^{\mu+\sigma_0(\mu)+\sigma_0^2(\mu)+\sigma_0^3(\mu)}\cdot\check T(\calO_E)=
{(t^{1/2})}^{\mu+\sigma_0(\mu)+\sigma_0^2(\mu)+\sigma_0^3(\mu)}\cdot\check T(\calO_E)=\]
\[={(t^{1/2})}^{2\mu+2\ct_0(\mu)}\cdot\check T(\calO_E).\]
The last equality follows from $\sigma_0=\on{Ad}_s\circ\ct_0$ with $s\in\check T^{\ct_0}$.
Now we can apply the same localization argument in Proposition \ref{cohomology} to prove (1).
Since $Sp_{2n}$ contains regular semisimple elements in $SL_{2n+1}$,
the Lemma \ref{S1-S3}
holds true for the embedding $\check\rho$, and 
Part (2) and (3) follow by the same argument 
as in the same argument as in the symmetric spaces setting. 
\end{proof}

\begin{remark}\label{complete}
(1) According to \cite[Remark 4.2.2]{DH},
the list in Example \ref{list 2} together 
with the cases $Spin_8(F)^{\ct}/Spin_8(\calO)^{\ct}$ and $SL_{2n+1}(E)^{\sigma}/SL_{2n+1}(\calO_E)^{\sigma}\cong SU_{2n+1}(F')/P$ cover all the 
twisted affine Grassmannians considered in \cite{Z}.

(2) In the forthcoming work, we will give an explicit description of the 
affine closures 
 $(\overline{PSO_8\times J_{G_2}/J_{PSO_8}|_{\fc_H}})$ and $\overline{(PGL_{2n+1}\times J_{SO_{2n+1}}/J_{PGL_{2n+1}}|_{\fc_{SO_{2n+1}}})}$ using the invariant theory of Vinberg pairs.

\end{remark}
\subsubsection{Subfield $F_n=\bC((t^n))$}
Let $G$ a complex reductive group and  $\check G$ the dual group. 
Let $F_n=\bC((t^n))\subset F=\bC((t))$
be the subfield
and $\calO_n=\bC[[t^n]]$ be its ring of integers.
Fix a primitive $n$-th root $\mu_n$ and let $\ct:\check G(F)\to \check G(F)$ is be order $n$ automorphism $\ct(\gamma)(t)=\gamma(\mu_nt)$.
Let $\check G(F)^{\ct}=\check G(F_n)$ and $\check G(\calO)^{\ct}=\check G(\calO_n)$
be the corresponding twisted loop and arc group and 
$\Gr^{\ct}=\check G(F_n)/\check G(\calO_n)$
the assoicated twisted affine Grassmannian.
Let $\check X_{\ct}(F)=\check G(F)/\check G(F_n)$.
Let $[n]:J_G\to J_G$ be the $n$-th
 power map $[n](g)=g^n$
 and $J_G[n]$ be the kernel of $[n]$.
 Consider the affine scheme
\[\overline{G\times J_G/J_{G}}=\on{Spec}(\bC[G\times J_G]^{J_{G}})\]
where $J_{G}$ acts on 
$G\times\fc_G$ via the natural embedding 
$J_G\to  G\times\fc_G$
and on $J_G$ via the map $[n]$.
 We have an isomorphism
\[\overline{G\times J_G/J_{G}}\cong\overline{G\times \fc_G/J_{G}[n]}=\on{Spec}(\bC[G\times\fc_G]^{J_{G}[n]})\]
where $J_{G}[n]$ acts on 
$G\times\fc_G$ via the natural embedding 
$J_{G}[n]\to J_G\to  G\times\fc_G$.

\begin{thm}
(1) There is commutative
diagram 
    \[\xymatrix{\on{Spec}(\on{H}_*^{\check G}(\check G(F)/\check G(\calO)))\ar[r]^{\rho^{\check G}}\ar[d]^\simeq&\on{Spec}(\on{H}_*^{\check G}(\check G(F_n)/\check G(\calO_n))\ar[d]^\simeq\\
J_{G}\ar[r]^{[n]}&J_{G}}\]
    where  $\rho^{\check G}_*$
    is induced by the embedding 
    $\rho:\check G(F_n)/\check G(\calO_n)\to\check G(F)/\check G(\calO)$.

(2)
There is an equivalence 
\[D^b_{\check G(\calO)}(\check G(F)/\check G(F_n))_e\cong 
\on{Perf}^{G}(\overline{G\times J_G/J_{G}})\cong\on{Perf}^{G}(\overline{G\times \fc_G/J_{G}[n]}).\]
Here $D^b_{\check G(\calO)}(\check G(F)/\check G(F_n))_e\subset D^b_{\check G(\calO)}(\check G(F)/\check G(F_n))$ is full subcategory 
generate by the dualizing sheaf $\omega_e$
of the $\check G(\calO)$-orbit through the base point $e\in \check G(F)/\check G(F_n)$
under the Hecke action.
\end{thm}
\begin{proof}
Since the
embedding 
$\check T(F_n)/\check T(\calO_n)\to \check T(F)^{}/\check T(\calO)^{}$
given by
$(t^n)^\lambda\to t^{n\lambda }\cdot\check T(\calO)$, 
the same 
proof as in Proposition \ref{cohomology},
implies part (1).

Note that Lemma \ref{S1-S3}
holds true for the embedding $\check\rho$ and 
the same argument as in the symmetric spaces setting implies (2).

\end{proof}

\begin{remark}
Let $X=G/K$ be the split symmetric spaces. 
Assume $K$ acts transitively on the set of regular 
nilpotent elements 
$\calN_{\fp^*,r}$ (e.g., in the case when $G$ is adjoint), 
then the same argument as in 
Proposition \ref{Wh=T^*X} implies 
there is a $G$-equivariant isomorphism
 $(\overline{G\times\fc_G/J_{G}[2]})\cong T^*X$ (in the proof, the adjoint assumption for $G$ is only used to insure that $\calN_{\fp^{*,r}}$ is a single orbit, see Remark \ref{single orbit}).
It will be nice to 
give a more explicit description of $(\overline{G\times\fc_G/J_{G}[n]})$
for general $n\geq 3$, see \cite[Remark 4.10]{De1} 
and \cite[Theorem 5.5.19 and Remark 5.5.20]{De2}
for interesting examples.
\end{remark}

\end{document}